\documentclass[11pt,a4paper]{article}

\usepackage[T1]{fontenc}
\usepackage[utf8]{inputenc}
\usepackage{lmodern}
\usepackage[margin=1in]{geometry}

\usepackage{amsmath,amssymb,mathtools}
\usepackage{amsthm}
\usepackage{mathrsfs}
\usepackage{xcolor}

\usepackage{graphicx}
\usepackage{subcaption}
\usepackage{float}

\usepackage{algorithm}
\usepackage[noend]{algpseudocode}

\usepackage{centernot}
\usepackage{ragged2e}
\usepackage{booktabs}
\usepackage{tabularx}
\usepackage{array}

\usepackage{tikz}
\usetikzlibrary{arrows.meta,positioning,fit,calc,backgrounds}

\usepackage[authoryear]{natbib}
\usepackage[colorlinks=true,linkcolor=blue,citecolor=blue,urlcolor=blue]{hyperref}
\usepackage{orcidlink}
\usepackage[nameinlink,capitalize]{cleveref}

\newcolumntype{L}[1]{>{\RaggedRight\arraybackslash}p{#1}}
\newcolumntype{C}[1]{>{\centering\arraybackslash}p{#1}}

\newtheorem{theorem}{Theorem}
\newtheorem{proposition}[theorem]{Proposition}
\newtheorem{example}{Example}
\newtheorem{remark}{Remark}
\newtheorem{lemma}{Lemma}
\newtheorem{corollary}{Corollary}
\newtheorem{definition}{Definition}

\title{A lattice-theoretic framework for hesitant fuzzy convexity beyond scalar observables}

\author{
Carlos Salvatierra\thanks{All authors contributed equally to this work. Corresponding author: \href{mailto:carlosss@ugr.es}{carlosss@ugr.es}.}\\
Department of Algebra, University of Granada, Granada 18071, Spain
\and
Pedro Huidobro\\
Department of Statistics and O.R., University of Oviedo, Oviedo 33003, Spain
\and
Raquel Fernandez-Peralta\\
Mathematical Institute, Slovak Academy of Sciences, Bratislava, Slovakia
}

\date{}

\begin{document}

\maketitle

\begin{abstract}
In fuzzy-set theory, the notion of convexity has often been formulated through scalar reductions, typically via scores and aggregation functions. Although useful, such reductions may obscure relevant order-theoretic information in the codomain, especially in complex set-valued settings. This article develops a lattice-theoretic framework for convexity on lattice-valued mappings over point-convex segment-generated abstract convexity spaces. The framework separates the segment structure of the domain from the lattice structure of the codomain, distinguishing intrinsic convexity, defined through the lattice meet, from relational and observable convexities induced by preorders or scalar maps. The article characterizes when scalar observables preserve or reconstruct intrinsic convexity, and which codomain operators preserve it through their pointwise extensions. It then specializes the framework to hesitant fuzzy sets endowed with the symmetric lattice, recovering classical fuzzy and interval-valued convexities as natural restrictions. The structural result shows that the symmetric order cannot be represented by any finite family of scalar observables. This obstruction appears even on two-point typical hesitant fuzzy elements. Consequently, symmetric hesitant convexity cannot, in general, be reconstructed by any finite family of scalar observables. Moreover, every finite family of monotone scalar observables admits a three-point hesitant profile that is scalar-convex for all selected descriptions but not symmetrically convex.

\end{abstract}


\section{Introduction}

Convexity is one of the basic structural notions in several branches of mathematics, including optimization, analysis, geometry, and decision theory. In fuzzy set theory, convexity emerged naturally from the seminal work of Zadeh \citep{Zadeh1965}. Several extensions of the concept have been proposed since then, in order to adapt it to richer models of uncertainty, such as interval-valued fuzzy sets (IVFSs) \citep{Sambuc75} or hesitant fuzzy sets (HFSs) \citep{Torra2010}. However, many of these proposals have been formulated in broad terms, without explicitly accounting for the specific structural features of the fuzzy-set model under consideration. Thus, a natural question arises: from a structural perspective, are the existing notions of convexity equally appropriate for all fuzzy-set models?

\smallskip

A substantial part of the literature on fuzzy convexity has traditionally focused on scalar descriptions of it. For instance, in the classical setting of \citet{Zadeh1965}, min-based fuzzy convexity was introduced for comparing membership degrees directly in the unit interval through the minimum operator. \citet{Diaz2017} generalized this idea by introducing an operator-based viewpoint through the notion of \(F\)-convexity, which replaces the minimum with a general operator \(F:[0,1]^2\to[0,1]\). In turn, in the hesitant fuzzy setting, \citet{RashidBeg2016} introduced a scalar-based notion of convexity for typical hesitant fuzzy sets (THFSs), namely fuzzy sets whose membership degrees are finite nonempty subsets of \([0,1]\), commonly called typical hesitant fuzzy elements (THFEs) \cite{Bedregal2014}. Their proposal relied on the arithmetic mean score function introduced by \citet{XuXia2011}. However, \citet{JanisMontesRencova2018} showed that this proposal was not stable under intersections in general, and \citet{Huidobro2021} addressed this limitation by proposing a general aggregation-based approach.

\smallskip

A different line of research has avoided the former preliminary scalarization by relying directly on order comparisons. In this approach, endpoint and intermediate values are compared through a suitable order on the codomain, rather than through a score or an aggregation function. Order-based convexities have been studied in particular for IVFSs \citep{Huidobro2022} and for THFSs \citep{Huidobro2025}.

\smallskip

Thus, the existing literature reveals two complementary lines of development: a scalar or observable approach, based on numerical summaries, and an order-theoretic approach, based on direct comparison. One of the aims of this article is to place both viewpoints within a common framework that may help determine when scalar observables can faithfully preserve the intrinsic order-theoretic structure. This clarification is especially relevant in set-valued contexts, where membership degrees are nonempty subsets of \([0,1]\) rather than single numbers. In such settings, the difficulty of working with suitable lattice structures has often encouraged the use of scores and aggregation functions, even though these scalar summaries may conceal structurally relevant information.

\smallskip

To address this issue, we study convexity on lattice-valued mappings over point-convex segment-generated abstract convexity spaces. This setting allows us to separate the geometric structure of the domain from the order-theoretic structure of the codomain, while recovering as particular cases ordinary segment convexity in real vector spaces, interval convexity on linearly ordered universes, and more general non-linearly ordered betweenness structures. On the codomain side, taken throughout to be a lattice $L$, we distinguish intrinsic convexity, defined through the lattice meet, from relational and observable convexities induced by preorders or scalar maps. We then determine when scalar observables preserve or reconstruct intrinsic convexity, and which codomain operators \(T:L^m\to L\) preserve it through their pointwise extensions on \(L^X\). These results show that preservation by scalar observables is governed by compatibility with finite meets as minima, exact scalar reconstruction requires order reflection, and preservation under pointwise induced operators depends on compatibility with the meet operation of the codomain lattice. The latter result follows the general framework of induced operators on lattices developed by \citet{LobilloMerinoNavarroSantos2021} and \citet{MerinoNavarroSantos2022}.

\smallskip

The main structural development of the article is obtained by specializing the previous framework to the symmetric lattice of HFSs. The relevance of this choice is reinforced by \citet{Rodriguez2016Position}, who identified as a foundational challenge the construction of HFS operations that recover the classical behaviour on standard fuzzy extensions while providing a lattice structure. In our setting, this structural perspective allows us to recover classical fuzzy and interval-valued convexities as natural restrictions of a common hesitant formulation. More importantly, by combining a reconstruction theorem for families of observables with an order-dimension obstruction, we show that the symmetric order contains finite subposets of arbitrarily large order dimension already within the class of two-point THFEs. Consequently, symmetric hesitant convexity cannot, in general, be globally reconstructed by any finite family of scalar observables. More concretely, every such family admits a three-point THFS profile that is scalar-convex for all selected descriptions but not symmetrically convex. This obstruction shows that the symmetric order cannot be globally recovered from finitely many scalar observables, in sharp contrast with chain-valued settings.

\smallskip

The article makes four main contributions. First, we formulate a lattice-theoretic framework for fuzzy convexity over point-convex segment-generated abstract convexity spaces, distinguishing an operational perspective based on codomain operators from a relational perspective based on preorders and scalar observables. This formulation provides a common perspective on several existing approaches to fuzzy convexity. Second, we establish the conditions under which scalar observables preserve or reconstruct intrinsic lattice convexity, and clarify its relation with lattice-order convexity. Third, we characterize the codomain operators whose pointwise extensions preserve intrinsic convexity. Finally, we specialize the framework to HFSs endowed with the symmetric lattice of \citet{JMNS22} and derive the finite scalar non-representability results that constitute the central structural contribution of the article.

\smallskip

The structure of this article is as follows. Section~\ref{sec:preliminaries} provides the necessary background and reviews the main convexity viewpoints introduced in the literature. Section~\ref{sec:convexity} introduces the operational and relational convexity perspectives on \(L\)-fuzzy sets. Section~\ref{section:bridges} studies the bridges between intrinsic and scalarized convexity. Section~\ref{sec:pointwise} characterizes the pointwise operators that preserve intrinsic convexity. Section~\ref{sec:hesitant} develops the hesitant specialization through the symmetric lattice and establishes the finite scalar non-representability results. Finally, Section~\ref{sec:conclusions} concludes the article and outlines future research directions.


\section{Preliminaries}\label{sec:preliminaries}

Throughout this article, we use the following mathematical background. A \emph{partially ordered set}, or \emph{poset}, is a pair \((P,\leq)\), where \(\leq\) is reflexive, antisymmetric, and transitive. If antisymmetry is omitted, one obtains a \emph{preordered set}. A preorder or order is called \emph{total} if, for any \(a,b\in P\), either \(a\leq b\) or \(b\leq a\). Given a poset \((P,\leq)\), a subset \(U\subseteq P\) is called an \emph{up-set} if \(a\in U\) and \(a\leq b\) imply \(b\in U\).

\smallskip

A \emph{lattice} is a poset \((L,\leq)\) such that every pair of elements \(a,b\in L\) has an infimum \(a\wedge b\) and a supremum \(a\vee b\). A lattice is \emph{bounded} if it has a least element \(0\) and a greatest element \(1\). It is \emph{complete} if every subset \(S\subseteq L\) has both an infimum \(\bigwedge S\) and a supremum \(\bigvee S\). In particular, every complete lattice is bounded, since \(\bigvee\varnothing\) and \(\bigwedge\varnothing\) are, respectively, the least and greatest elements of \(L\).

\smallskip

In this article, we work inside the general framework of lattice-valued fuzzy sets. If \((L,\leq)\) is a lattice and \(X\) is a universe, we denote by \(L^X\) the class of all mappings from \(X\) to \(L\). A map \(A\in L^X\) is called an \(L\)-fuzzy set \citep{Goguen1967}. The pointwise order on \(L^X\) is defined by
\[
A\sqsubseteq B \iff A(x)\leq B(x)\ \text{for all }x\in X,
\]
and the pointwise lattice operations are given by
\[
(A\sqcap B)(x):=A(x)\wedge B(x),
\qquad
(A\sqcup B)(x):=A(x)\vee B(x),
\qquad x\in X.
\]
For instance, taking \(L=[0,1]\), endowed with the usual order, one recovers classical fuzzy sets on \(X\) \citep{Zadeh1965}, i.e., a fuzzy set is a mapping \(A:X\to[0,1]\). Likewise, by choosing \(L=I([0,1])\), the family of closed intervals of \([0,1]\) endowed with the usual interval order, one obtains standard interval-valued fuzzy sets \citep{Bustince2016,GrattanGuinness1976,Jahn1975,Sambuc75,Zadeh1971}.


\subsection{Related work on fuzzy convexity}

\smallskip

The development of fuzzy convexity can be traced back to the classical notion of convexity for crisp subsets of a linear space \citep{Klee1971}. A subset \(C\) of a linear space \(X\) is convex if, for every \(x,y\in C\) and every \(\lambda\in[0,1]\), one has \(\lambda x+(1-\lambda)y\in C\). Given a fuzzy set \(A:X\to[0,1]\), one possible arithmetic way to transfer this idea is to impose the Jensen-type condition considered, for instance, by \citet{AmmarMetz1992},
\begin{equation}\label{eq:jensen-fuzzy-convexity}
A(\lambda x+(1-\lambda)y)
\geq
\lambda A(x)+(1-\lambda)A(y).
\end{equation}
Although natural, this formulation depends on the arithmetic structure of the codomain and is therefore not directly available in general lattice-valued settings.

\smallskip

The standard cut-compatible formulation in fuzzy convexity is the min-based condition introduced by \citet{Zadeh1965} in the classical fuzzy setting,
\begin{equation}\label{eq:min-based-fuzzy-convexity}
A(\lambda x+(1-\lambda)y)
\geq
\min\{A(x),A(y)\}.
\end{equation}
This definition is stable under intersections and is equivalent to the convexity of all \(\alpha\)-cuts,
\begin{equation}\label{eq:alpha-cuts-classical}
[A]_\alpha=\{x\in X:A(x)\geq \alpha\},
\qquad \alpha\in(0,1].
\end{equation}
This cut-compatible viewpoint became one of the basic foundations of convex fuzzy-set theory; see, for instance, the systematic treatment of \citet{Lowen1980}.

\smallskip

A direct lattice-valued counterpart of this min-based condition was considered by \citet{Maruyama2009} in a Euclidean context. For a completely distributive lattice \(L\), an \(L\)-fuzzy set \(A:\mathbb{R}^n\to L\) is said to be \(L\)-fuzzy convex when
\begin{equation}\label{eq:maruyama-l-fuzzy-convexity}
A(\lambda x+(1-\lambda)y)
\geq
A(x)\wedge A(y)
\end{equation}
for all \(x,y\in\mathbb{R}^n\) and every \(\lambda\in[0,1]\). Moreover, \citet{Maruyama2009} showed that this condition is equivalent to the convexity of all lattice-valued level cuts
\(
\{x\in\mathbb{R}^n:a\leq A(x)\},
~ a\in L.
\)
This places the classical cut-compatible viewpoint within a meet-based lattice-valued setting and anticipates the intrinsic convexity that we consider below.

\smallskip

An operator-oriented generalization of min-based convexity was introduced by \citet{Diaz2017} through the notion of \(F\)-convexity. Given a binary operator \(F:[0,1]^2\to[0,1]\), a fuzzy set \(A:X\to[0,1]\) is said to be \(F\)-convex when
\begin{equation}\label{eq:F-convexity}
A(y)\geq F(A(x),A(z))
\end{equation}
for all \(x\leq y\leq z\), \(x,y,z\in X\). The usual min-based fuzzy convexity is recovered when \(F=\min\). This approach shifts the focus from a fixed convexity inequality to the choice of a codomain operator.

\smallskip

In hesitant fuzzy settings, the study of convexity has mainly focused on scalar summaries of typical hesitant information. \citet{RashidBeg2016}, for instance, proposed a score-based approach for typical hesitant fuzzy sets (THFSs). We denote by \(F^*([0,1])\) the family of all nonempty finite subsets of \([0,1]\). The membership degrees of a THFS are elements of \(F^*([0,1])\), commonly called typical hesitant fuzzy elements (THFEs) \cite[Def.~3.1]{Bedregal2014}. Given a THFS \(A:X\to F^*([0,1])\), they used the arithmetic mean score function \citep{Farhadinia2014,XuXia2011}
\begin{equation}\label{eq:arithmetic-mean-score}
s_A(x)=\frac{1}{|A(x)|}\sum_{\gamma\in A(x)}\gamma,
\end{equation}
and defined convexity by
\begin{equation}\label{eq:score-based-thfs-convexity}
s_A(\lambda x+(1-\lambda)y)
\geq
\min\{s_A(x),s_A(y)\}.
\end{equation}

This definition preserves a cut-based interpretation when the cuts are given by \([A]^s_\alpha=\{x\in X:s_A(x)\geq \alpha\}\). However, \citet{JanisMontesRencova2018} showed that it is not stable under the usual intersection of THFSs. This loss of stability is relevant, since closure under intersections is a basic structural property of the classical fuzzy approach.

\smallskip

To address this drawback, \citet{JanisMontesRencova2018} introduced a max-based notion of convexity for THFSs. Given a THFS \(A\), and \(A_{\max}(x):=\max A(x)\), the corresponding condition is given by
\begin{equation}\label{eq:max-based-thfs-convexity}
A_{\max}(\lambda x+(1-\lambda)y)
\geq
\min\{A_{\max}(x),A_{\max}(y)\}.
\end{equation}

The aggregation-based approach of \citet{Huidobro2021} subsequently generalized this idea. Given an aggregation function \(\mathcal S\), a THFS A over a totally ordered universe is said to be \(\mathcal S\)-convex if
\begin{equation}\label{eq:aggregation-based-thfs-convexity}
\mathcal S(A(y))
\geq
\min\{\mathcal S(A(x)),\mathcal S(A(z))\}
\end{equation}
for every \(x\leq y\leq z\),~~\(x,y,z\in X\). This notion recovers the arithmetic mean score approach of Rashid and Beg and the max-based approach as particular cases. However, all these formulations still reduce each typical hesitant value to a scalar quantity.

\smallskip

To avoid these preliminary reductions, some order-based approaches have also been proposed. These attempts have focused, in particular, on IVFSs \cite{Huidobro2022} and on THFSs \cite{Huidobro2025}. In these, convexity is imposed directly through a codomain order \(\preccurlyeq\). Thus, given a totally ordered universe \(X\), for an IVFS or a THFS \(A\), the condition is of the form
\begin{equation}\label{eq:order-based-convexity-literature}
A(x)\preccurlyeq A(y)
\quad\text{or}\quad
A(z)\preccurlyeq A(y),
\qquad x<y<z, \qquad x,y,z\in X.
\end{equation}
In other words, the intermediate value must dominate at least one endpoint value with respect to the chosen order. In this way, the comparison is carried out directly at the level of the codomain values, rather than after a scalar aggregation.

\smallskip

Most of the previous notions are formulated either on real vector spaces 
or on linearly ordered universes. Although these domains may be natural in certain concrete contexts, they become restrictive whenever the relevant notion of intermediate point is not encoded by linear combinations or by a total order, as occurs, for instance, in graph-like, geodesic, or more general betweenness structures. This is the reason why we work inside point-convex segment-generated abstract convexity spaces instead. This setting allows us to treat the classical segment \(\lambda x+(1-\lambda)y\), the order interval \([x,z]\), and other betweenness structures through a common segment mechanism.

\smallskip

A complementary structural perspective on lattice-valued convexity has been developed through the theory of \(L\)-convex spaces. In particular, \citet{PangShi2017} studied \(L\)-convex structures generated from convexity of level cuts. This viewpoint is related to the cut-based characterization of intrinsic convexity considered later in this article. Our focus, however, is different. Rather than constructing or classifying convex structures on \(L^X\), we fix a segment-generated abstract convexity on the domain and study concrete convexity notions for individual lattice-valued mappings from \(X\) to \(L\), together with their comparison, preservation, and reconstruction properties. 



\section{Two convexity perspectives on \(L\)-fuzzy sets}\label{sec:convexity}

In this section, we introduce the two new convexity schemes. The purpose is to separate two ingredients that are often mixed in fuzzy convexity: the segment structure of the domain and the order-theoretic structure of the codomain. We begin by fixing the abstract notion of segment used on the domain.

\smallskip

Following \citep{vandeVel1993}, an \emph{abstract convexity space} is a pair \((X,\Gamma)\), where \(X\) is a nonempty set and \(\Gamma\subseteq\mathcal P(X)\) is a family of subsets of \(X\), called convex sets, such that \(\varnothing,X\in\Gamma\), \(\Gamma\) is closed under arbitrary intersections, and \(\Gamma\) is closed under nested unions. For every \(S\subseteq X\), its convex hull is defined by
\begin{equation*}
\operatorname{co}_{\Gamma}(S):=\bigcap\{C\in\Gamma:S\subseteq C\}.
\end{equation*}

In this article, we restrict our attention to \emph{point-convex segment-generated abstract convexity spaces}. More precisely, we assume that \(\operatorname{co}_{\Gamma}(\{x\})=\{x\}\) for every \(x\in X\), and that a subset \(C\subseteq X\) belongs to \(\Gamma\) if and only if
\begin{equation*}
x,z\in C \Longrightarrow \operatorname{co}_{\Gamma}(\{x,z\})\subseteq C.
\end{equation*}
We write
\begin{equation*}
[x,z]_{\Gamma}:=\operatorname{co}_{\Gamma}(\{x,z\}).
\end{equation*}
Thus, in a segment-generated setting, a subset \(C\subseteq X\) is convex if and only if \(x,z\in C\) imply \([x,z]_{\Gamma}\subseteq C\). This agrees with the interval-convexity viewpoint in betweenness structures \citep{AndersonBankstonMcCluskey2021}, where convexity is determined by requiring that every interval between two points of a set remains inside the set. For instance, this setting includes the linearly ordered universes usually considered in the literature, by taking \([x,z]_{\Gamma}\) as the order interval between \(x\) and \(z\), and real vector spaces, by taking \([x,z]_{\Gamma}\) as the usual segment joining \(x\) and \(z\). It also allows non-linearly ordered domains, as illustrated below. From now on, \((X,\Gamma)\) will denote a point-convex segment-generated abstract convexity space.

\begin{remark}\label{rem:why-segment-generated}
The segment-generated assumption ensures that the local conditions formulated through points \(y\in[x,z]_{\Gamma}\) characterize \(\Gamma\)-convexity. Without this assumption, they would only guarantee closure under two-point segments, which need not imply membership in \(\Gamma\) for a general abstract convexity space. The point-convex condition ensures that \([x,x]_{\Gamma}=\{x\}\) for every \(x\in X\).
\end{remark}


\subsection{Operational perspective: \((L,\otimes)\)-convexity}

\smallskip

The first perspective of convexity is operational. This vision abstracts the role played by the codomain operator in min-based fuzzy convexity and in its operator-based generalizations.

\begin{definition}\label{defLconvex}
Let \((X,\Gamma)\) be a point-convex segment-generated abstract convexity space, let \((L,\leq)\) be a partially ordered set, and let \(\otimes:L\times L\to L\) be a binary operator. A mapping \(A:X\to L\) is said to be \emph{\(\otimes\)-convex} if
\begin{equation*}
A(y)\ge A(x)\otimes A(z)
\end{equation*}
for all \(x,z\in X\) and every \(y\in[x,z]_{\Gamma}\).
\end{definition}

\begin{remark}
When \(L=[0,1]\), \(\otimes=F\), and the domain convexity is induced by a linear order, Definition~\ref{defLconvex} recovers the condition in \eqref{eq:F-convexity}. In particular, the usual min-based fuzzy convexity in \eqref{eq:min-based-fuzzy-convexity} is obtained by taking \(F=\min\). The standard segment-based formulation is also recovered when \([x,z]_{\Gamma}\) is the segment joining \(x\) and \(z\) in a real vector space.
\end{remark}

\smallskip

In a lattice-valued context, the canonical operational choice is \(\otimes=\wedge\). This gives the intrinsic notion used throughout the article. Unlike scalar or aggregation-based approaches, this definition does not reduce the codomain value to a numerical observable.

\subsubsection{Intrinsic lattice convexity}

\smallskip

\begin{definition}\label{defintrinsic}
Let \((X,\Gamma)\) be a point-convex segment-generated abstract convexity space, and let \((L,\leq,\wedge,\vee)\) be a lattice. A mapping \(A:X\to L\) is said to be \emph{intrinsically convex} if
\begin{equation*}
A(y)\ge A(x)\wedge A(z)
\end{equation*}
for all \(x,z\in X\) and every \(y\in[x,z]_{\Gamma}\).
\end{definition}

The following example shows that this definition is not restricted to linearly ordered domains.

\begin{example}\label{ex:non-chain-domain}
Let \(X=\{a,b,c,d\}\) be the vertex set of the cycle graph \(C_4\) of Figure~\ref{fig:non-chain-convexity}, with edges
\begin{equation*}
a-b,\quad b-c,\quad c-d,\quad d-a.
\end{equation*}
We consider \(X\) endowed with the geodesic convexity induced by the graph \citep{FarberJamison1987}. Thus, for \(x,z\in X\), the segment \([x,z]_{\Gamma}\) is the set of all vertices lying on shortest paths between \(x\) and \(z\). In particular,
\begin{equation*}
[a,b]_{\Gamma}=\{a,b\},\quad [b,c]_{\Gamma}=\{b,c\},\quad
[c,d]_{\Gamma}=\{c,d\},\quad [d,a]_{\Gamma}=\{d,a\},
\end{equation*}
whereas
\begin{equation*}
[a,c]_{\Gamma}=X,\qquad [b,d]_{\Gamma}=X.
\end{equation*}
Hence \((X,\Gamma)\) is a point-convex segment-generated abstract convexity space, but it is not induced by a linear order.

Let \(L=[0,1]\) with its usual lattice structure, and let \(A:X\to[0,1]\) be a fuzzy set given by
\begin{equation*}
A(a)=0.2,\qquad A(b)=0.2,\qquad A(c)=0.8,\qquad A(d)=0.6.
\end{equation*}
Then \(A\) is intrinsically convex. Indeed, for adjacent vertices the segment contains only the two endpoints, so the condition is immediate. For the opposite pair \(a,c\), one has \(\min\{A(a),A(c)\}=0.2\), and all vertices have membership degree at least \(0.2\). The same argument applies to the opposite pair \(b,d\), since \(\min\{A(b),A(d)\}=0.2\). Therefore,
\begin{equation*}
A(y)\geq \min\{A(x),A(z)\}
\end{equation*}
for all \(x,z\in X\) and every \(y\in[x,z]_{\Gamma}\).

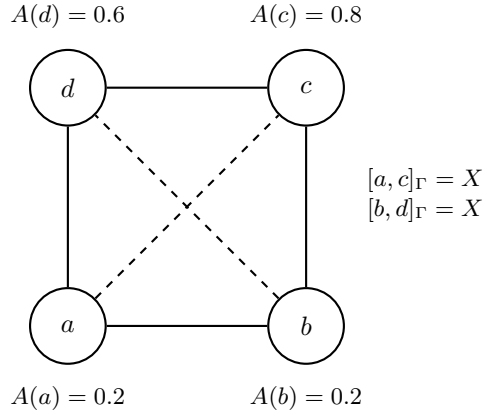
\begin{figure}[htbp]
\centering
\begin{tikzpicture}[
    scale=1.05,
    every node/.style={font=\small},
    vertex/.style={circle,draw,thick,minimum size=10mm,inner sep=0pt,fill=white},
    value/.style={font=\footnotesize},
    edge/.style={thick},
    diag/.style={dashed,thick}
]

\node[vertex] (a) at (0,0) {$a$};
\node[vertex] (b) at (3,0) {$b$};
\node[vertex] (c) at (3,3) {$c$};
\node[vertex] (d) at (0,3) {$d$};

\draw[edge] (a)--(b)--(c)--(d)--(a);

\draw[diag] (a)--(c);
\draw[diag] (b)--(d);

\node[value,below=4pt] at (a.south) {$A(a)=0.2$};
\node[value,below=4pt] at (b.south) {$A(b)=0.2$};
\node[value,above=4pt] at (c.north) {$A(c)=0.8$};
\node[value,above=4pt] at (d.north) {$A(d)=0.6$};

\node[align=center,font=\footnotesize] at (4.5,1.65)
{\([a,c]_{\Gamma}=X\)\\ \([b,d]_{\Gamma}=X\)};

\end{tikzpicture}
\caption{A convex fuzzy set on a non-linearly ordered domain.}
\label{fig:non-chain-convexity}
\end{figure}
\end{example}

We next record the cut characterization of intrinsic convexity. In the Euclidean setting, \citet{Maruyama2009} showed that meet-based \(L\)-fuzzy convexity can be characterized through the convexity of all lattice-valued level cuts. Related cut-based analyses for lattice- and poset-valued convex mappings were later developed by \citet{JanisSeseljaTepavcevic2017}. The result below provides the corresponding characterization in our framework of point-convex segment-generated abstract convexity spaces.

\begin{definition}\label{defprincut}
Let \((X,\Gamma)\) be a point-convex segment-generated abstract convexity space, let \((L,\leq)\) be a partially ordered set, let \(A:X\to L\), and let \(a\in L\). The \emph{principal cut} of \(A\) at level \(a\) is
\begin{equation*}
[A]_a:=\{x\in X:\ a\le A(x)\}.
\end{equation*}
\end{definition}

\begin{theorem}\label{teo_princ_cuts}
Let \((X,\Gamma)\) be a point-convex segment-generated abstract convexity space, let \((L,\leq,\wedge,\vee)\) be a lattice, and let \(A:X\to L\). Then the following statements are equivalent:
\begin{enumerate}
    \item \(A\) is intrinsically convex;
    \item for every \(a\in L\), the principal cut \([A]_a\) is convex, that is, \([A]_a\in\Gamma\).
\end{enumerate}
\end{theorem}

\begin{proof}
Assume first that \(A\) is intrinsically convex. Let \(a\in L\), and let \(x,z\in[A]_a\). Then \(a\le A(x)\) and \(a\le A(z)\), hence \(a\le A(x)\wedge A(z)\). If \(y\in[x,z]_{\Gamma}\), intrinsic convexity gives \(A(y)\ge A(x)\wedge A(z)\ge a\), so \(y\in[A]_a\). Thus \([x,z]_{\Gamma}\subseteq[A]_a\). Since this holds for all \(x,z\in[A]_a\), the segment-generated property implies that \([A]_a\) is convex.

Conversely, assume that every principal cut is convex. Let \(x,z\in X\) and let \(y\in[x,z]_{\Gamma}\). Set \(a:=A(x)\wedge A(z)\). Then \(a\le A(x)\) and \(a\le A(z)\), so \(x,z\in[A]_a\). Since \([A]_a\) is convex, we have \([x,z]_{\Gamma}\subseteq[A]_a\), and therefore \(y\in[A]_a\). Hence \(A(y)\ge a=A(x)\wedge A(z)\), and \(A\) is intrinsically convex.
\end{proof}


\begin{remark}\label{rem:intrinsic-induced-L-convex}
Theorem~\ref{teo_princ_cuts} shows that intrinsic convexity admits a complete cut-based description. Thus, a mapping \(A:X\to L\) is intrinsically convex if and only if all its principal cuts belong to \(\Gamma\). This connects our notion with the cut-generated viewpoint developed by \citet[Proposition~5.2]{PangShi2017}, where families of \(L\)-valued mappings are selected through the convexity of their level cuts and studied as \(L\)-convex structures. Our focus, however, is different. Rather than constructing an \(L\)-convex structure on \(L^X\), we fix a point-convex segment-generated abstract convexity \(\Gamma\) on the domain and use principal cuts to characterize intrinsic convexity for individual mappings \(A:X\to L\).
\end{remark}


\begin{example}\label{ex1}
Let \(X=\{x_1<x_2<x_3<x_4<x_5\}\), endowed with its usual order convexity, that is,
\begin{equation*}
[x_i,x_j]_{\Gamma}:=\{x_k:\min\{i,j\}\le k\le \max\{i,j\}\}.
\end{equation*}
Let \(L=I([0,1])\) be endowed with the usual interval order
\begin{equation*}
[a,b]\le_I[c,d]\iff a\le c\ \text{and}\ b\le d,
\end{equation*}
and define an IVFS \(A:X\to I([0,1])\) by
\begin{equation*}
A(x_1)=[0.10,0.30],\quad
A(x_2)=[0.35,0.55],\quad
A(x_3)=[0.70,0.90],
\end{equation*}
\begin{equation*}
A(x_4)=[0.50,0.75],\quad
A(x_5)=[0.20,0.45].
\end{equation*}
Then \(A\) is intrinsically convex. Indeed, its lower and upper endpoint profiles are both convex in the usual min-based sense. Equivalently, all its principal cuts are convex subsets of \(X\). For instance,
\begin{equation*}
[A]_{[0.35,0.55]}=\{x_2,x_3,x_4\},
\qquad
[A]_{[0.65,0.85]}=\{x_3\}.
\end{equation*}
Hence, by Theorem~\ref{teo_princ_cuts}, \(A\) is intrinsically convex. Figure~\ref{fig:level-cuts-example} illustrates this example.

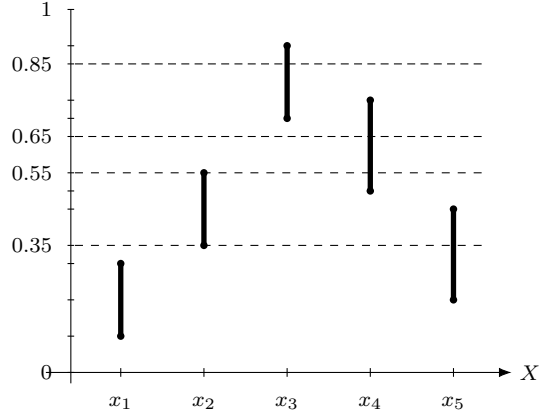
\begin{figure}[htbp]
\centering
\scriptsize
\begin{tikzpicture}[x=1.1cm,y=4.8cm,>=Latex,font=\scriptsize]

\draw[->] (-0.3,0) -- (5.3,0) node[right] {$X$};
\draw (0,-0.03) -- (0,1);

\foreach \y in {0,0.10,0.20,0.30,0.35,0.45,0.50,0.55,0.65,0.70,0.75,0.85,0.90,1}{
  \draw (-0.04,\y)--(0.04,\y);
}
\node[left=4pt] at (0,0) {$0$};
\node[left=4pt] at (0,0.35) {$0.35$};
\node[left=4pt] at (0,0.55) {$0.55$};
\node[left=4pt] at (0,0.65) {$0.65$};
\node[left=4pt] at (0,0.85) {$0.85$};
\node[left=4pt] at (0,1) {$1$};

\foreach \x/\lab in {0.6/x_1,1.6/x_2,2.6/x_3,3.6/x_4,4.6/x_5}{
  \draw (\x,-0.012)--(\x,0.012);
  \node[below=6pt] at (\x,0) {$\lab$};
}

\draw[dashed] (0.05,0.35)--(5.0,0.35);
\draw[dashed] (0.05,0.55)--(5.0,0.55);
\draw[densely dashed] (0.05,0.65)--(5.0,0.65);
\draw[densely dashed] (0.05,0.85)--(5.0,0.85);

\foreach \x/\a/\b in {
  0.6/0.10/0.30,
  1.6/0.35/0.55,
  2.6/0.70/0.90,
  3.6/0.50/0.75,
  4.6/0.20/0.45}{
  \draw[line width=2pt] (\x,\a)--(\x,\b);
  \fill (\x,\a) circle (1.4pt);
  \fill (\x,\b) circle (1.4pt);
}

\end{tikzpicture}
\caption{Convexity of an interval-valued fuzzy set.}
\label{fig:level-cuts-example}
\end{figure}
\end{example}

\begin{remark}\label{rem:closure-intrinsic}
The cut characterization also gives the standard closure behaviour of intrinsic convexity under infima. If \(L\) is complete, intrinsically convex \(L\)-fuzzy sets are closed under arbitrary pointwise infima, since principal cuts of an infimum are intersections of principal cuts and \(\Gamma\) is closed under arbitrary intersections.

A corresponding closure property for suprema of chains requires an additional compatibility condition on the lattice. Assume that, for every pair of chains \((a_i)_{i\in I}\) and \((b_i)_{i\in I}\) in \(L\),
\begin{equation*}
\bigvee_{i\in I}(a_i\wedge b_i)=
\left(\bigvee_{i\in I}a_i\right)\wedge
\left(\bigvee_{i\in I}b_i\right).
\end{equation*}
Then intrinsically convex \(L\)-fuzzy sets are closed under pointwise suprema of chains. Indeed, let \((A_i)_{i\in I}\) be a chain of intrinsically convex mappings and define \(A:=\bigvee_{i\in I}A_i\) pointwise. For \(x,z\in X\) and \(y\in[x,z]_{\Gamma}\), one has
\begin{equation*}
A(y)
=
\bigvee_{i\in I}A_i(y)
\geq
\bigvee_{i\in I}\bigl(A_i(x)\wedge A_i(z)\bigr)
=
A(x)\wedge A(z).
\end{equation*}
Hence \(A\) is intrinsically convex.
\end{remark}

\subsection{Relational perspective: preorder-based convexity}

\smallskip

The second perspective on convexity is relational. This viewpoint uses a comparison relation on the codomain, instead of a binary operator. As we show, this is the natural setting for order-based and scalar approaches to convexity. 

\begin{definition}\label{def_order_convex}
Let \((X,\Gamma)\) be a point-convex segment-generated abstract convexity space, let \((L,\preccurlyeq)\) be a preordered set, and let \(A:X\to L\). We say that \(A\) is \emph{\(\preccurlyeq\)-convex} if, for every up-set \(U\subseteq L\), the inverse image \(A^{-1}(U)\) is convex, that is, \(A^{-1}(U)\in\Gamma\).
\end{definition}


This definition provides an order-theoretic analogue of the classical cut-based interpretation of fuzzy convexity in \eqref{eq:alpha-cuts-classical} and of the lattice-valued level-cut viewpoint in \eqref{eq:maruyama-l-fuzzy-convexity}. Moreover, when the domain convexity is induced by a linear order, it recovers the order-based condition in \eqref{eq:order-based-convexity-literature}. The following elementary characterization gives its local form.

\begin{theorem}\label{teo_orden_local}
Let \((X,\Gamma)\) be a point-convex segment-generated abstract convexity space, let \((L,\preccurlyeq)\) be a preordered set, and let \(A:X\to L\). Then the following statements are equivalent:
\begin{enumerate}
    \item \(A\) is \(\preccurlyeq\)-convex;
    \item for every \(x,z\in X\) and every \(y\in[x,z]_{\Gamma}\), one has
    \[
    A(x)\preccurlyeq A(y)
    \quad\text{or}\quad
    A(z)\preccurlyeq A(y).
    \]
\end{enumerate}
\end{theorem}

\begin{proof}
Assume first that \(A\) is \(\preccurlyeq\)-convex. Fix \(x,z\in X\) and \(y\in[x,z]_{\Gamma}\), and define
\[
U:=\uparrow A(x)\cup \uparrow A(z),
\qquad
\uparrow a:=\{b\in L:\ a\preccurlyeq b\}.
\]
Then \(U\) is an up-set, and \(x,z\in A^{-1}(U)\). Since \(A^{-1}(U)\) is convex, we get \([x,z]_{\Gamma}\subseteq A^{-1}(U)\). Hence \(y\in A^{-1}(U)\), and therefore \(A(y)\in \uparrow A(x)\cup\uparrow A(z)\). Thus \(A(x)\preccurlyeq A(y)\) or \(A(z)\preccurlyeq A(y)\).

Conversely, assume (2), and let \(U\subseteq L\) be an up-set. Take \(x,z\in A^{-1}(U)\), and let \(y\in[x,z]_{\Gamma}\). Since \(A(x),A(z)\in U\), condition (2) gives either \(A(x)\preccurlyeq A(y)\) or \(A(z)\preccurlyeq A(y)\). As \(U\) is an up-set, in both cases \(A(y)\in U\). Therefore \([x,z]_{\Gamma}\subseteq A^{-1}(U)\). Since this holds for all \(x,z\in A^{-1}(U)\), the segment-generated property implies that \(A^{-1}(U)\) is convex.
\end{proof}

If \((L,\preccurlyeq)\) is totally preordered, then \(\preccurlyeq\)-convexity can be written as
\[
\min\nolimits_{\preccurlyeq}\{A(x),A(z)\}\preccurlyeq A(y),
\qquad y\in[x,z]_{\Gamma},
\]
where \(\min_{\preccurlyeq}\) denotes a minimal element, with respect to \(\preccurlyeq\), among \(A(x)\) and \(A(z)\).

\begin{example}\label{ex2}
Let \(X=\{x_1<x_2<x_3\}\), endowed with its usual order convexity, and let \(L=\{0,p,q,1\}\) be the four-element diamond lattice, where \(0\preccurlyeq p\preccurlyeq 1\), \(0\preccurlyeq q\preccurlyeq 1\), and \(p\) and \(q\) are incomparable \citep{Gratzer2011}. Define an L-fuzzy set \(A:X\to L\) by
\[
A(x_1)=p,\qquad A(x_2)=1,\qquad A(x_3)=q.
\]
Then \(A\) is \(\preccurlyeq\)-convex. Indeed, for the only nontrivial segment, \(x_2\in[x_1,x_3]_{\Gamma}\), one has
\[
A(x_1)=p\preccurlyeq 1=A(x_2),
\qquad
A(x_3)=q\preccurlyeq 1=A(x_2).
\]
Thus the local condition in Theorem~\ref{teo_orden_local} is satisfied.

Equivalently, one may check the definition directly. The up-sets of \(L\) have inverse images among \(\varnothing\), \(\{x_2\}\), \(\{x_1,x_2\}\), \(\{x_2,x_3\}\), and \(X\), all of which are convex subsets of \(X\). Figure~\ref{fig:upset-preimage-example} illustrates this situation.

\begin{figure}[htbp]
\centering
\begin{tikzpicture}[
    >=Latex,
    font=\small,
    every node/.style={align=center},
    point/.style={circle,fill=black,inner sep=2.2pt},
    lnode/.style={circle,draw,minimum size=8.5mm,inner sep=0pt,fill=white},
    upsetnode/.style={circle,draw,fill=gray!15,minimum size=8.5mm,inner sep=0pt}
]

\node[point,label=below:$x_1$] (x1) at (0,0) {};
\node[point,label=below:$x_2$] (x2) at (1.7,0) {};
\node[point,label=below:$x_3$] (x3) at (3.4,0) {};

\draw[thick] (x1)--(x2)--(x3);
\node[below=13pt] at (1.7,0) {\emph{domain} \(X\)};

\node[lnode]     (zero) at (7.7,0) {$0$};
\node[upsetnode] (p)    at (6.8,1.35) {$p$};
\node[lnode]     (q)    at (8.6,1.35) {$q$};
\node[upsetnode] (one)  at (7.7,2.75) {$1$};

\draw[thick] (zero)--(p)--(one);
\draw[thick] (zero)--(q)--(one);

\node[below=13pt] at (7.7,0) {\emph{codomain} \(L\)};

\node[right=16pt] at (8.55,2.2) {$\uparrow p=\{p,1\}$};

\draw[dashed,->,shorten >=4pt,shorten <=3pt] (x1) to[bend left=12] (p.west);
\draw[dashed,->,shorten >=4pt,shorten <=3pt] (x2) to[bend left=18] (one.west);
\draw[dashed,->,shorten >=4pt,shorten <=3pt] (x3) to[bend left=8]  (q.west);

\node[above=1.35cm] at (1.45,0) {$A^{-1}(\uparrow p)=\{x_1,x_2\}$};

\end{tikzpicture}
\caption{Preimage of an up-set under an \(L\)-fuzzy set.}
\label{fig:upset-preimage-example}
\end{figure}
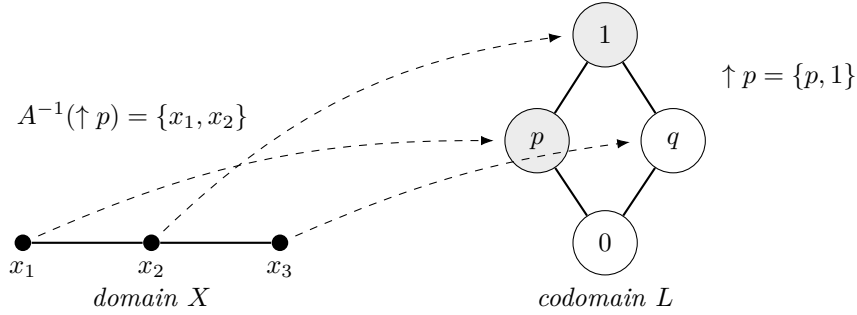
\end{example}

\smallskip

\smallskip

A particular case is obtained by taking the preorder to be the lattice order.

\begin{definition}\label{def_lattice_order_convex}
Let \((L,\le,\wedge,\vee)\) be a lattice. When the preorder \(\preccurlyeq\) in Definition~\ref{def_order_convex} is the lattice order \(\le\), we call the resulting notion \emph{lattice-order convexity}.
\end{definition}

A different specialization of Definition~\ref{def_order_convex} arises when the comparison among codomain values is induced not by the lattice order itself, but by a scalar map into \([0,1]\).


\subsubsection{Observable convexity}

\smallskip

In the following paragraphs, we turn to the scalar case. Several notions of fuzzy convexity rely on numerical summaries of the codomain values. In the present framework, we regard all these maps as general observables. We adopt this term to make a clear distinction from particular constructions such as score or aggregation functions. Indeed, direct inspection shows, for instance, that every score function is an observable, whereas an observable need not satisfy the specific axioms usually required of scores \citep{Alcantud}.

\begin{definition}\label{def_observable_convex}
Let \((X,\Gamma)\) be a point-convex segment-generated abstract convexity space, let \(L\) be a nonempty set, let \(\varphi:L\to[0,1]\) be a map, and let \(A:X\to L\). For \(\alpha\in[0,1]\), define
\[
[A]^\varphi_\alpha:=\{x\in X:\ \varphi(A(x))\ge \alpha\}.
\]
We say that \(A\) is \emph{\(\varphi\)-convex} if \([A]^\varphi_\alpha\) is convex for every \(\alpha\in[0,1]\). The map \(\varphi\) will be called an \emph{observable}.
\end{definition}

The following characterization is the scalar analogue of Theorem~\ref{teo_orden_local}.

\begin{theorem}\label{teo_obs_local}
Let \((X,\Gamma)\) be a point-convex segment-generated abstract convexity space, let \(L\) be a nonempty set, let \(\varphi:L\to[0,1]\), and let \(A:X\to L\). Then the following statements are equivalent:
\begin{enumerate}
    \item \(A\) is \(\varphi\)-convex;
    \item for every \(x,z\in X\) and every \(y\in[x,z]_{\Gamma}\), one has
    \[
    \varphi(A(y))\ge \min\{\varphi(A(x)),\varphi(A(z))\}.
    \]
\end{enumerate}
\end{theorem}

\begin{proof}
Assume first that \(A\) is \(\varphi\)-convex. Let \(x,z\in X\), let \(y\in[x,z]_{\Gamma}\), and define
\[
\alpha:=\min\{\varphi(A(x)),\varphi(A(z))\}.
\]
Then \(x,z\in[A]^\varphi_\alpha\). Since this cut is convex, we obtain \([x,z]_{\Gamma}\subseteq[A]^\varphi_\alpha\), and therefore \(y\in[A]^\varphi_\alpha\). Hence
\[
\varphi(A(y))\ge \alpha=\min\{\varphi(A(x)),\varphi(A(z))\}.
\]

Conversely, assume that the second statement holds, and let \(\alpha\in[0,1]\). Take \(x,z\in[A]^\varphi_\alpha\), and let \(y\in[x,z]_{\Gamma}\). Then \(\varphi(A(x))\ge\alpha\) and \(\varphi(A(z))\ge\alpha\). By the second statement,
\[
\varphi(A(y))\ge \min\{\varphi(A(x)),\varphi(A(z))\}\ge\alpha,
\]
so \(y\in[A]^\varphi_\alpha\). Therefore \([x,z]_{\Gamma}\subseteq[A]^\varphi_\alpha\). Since this holds for all \(x,z\in[A]^\varphi_\alpha\), the segment-generated property implies that \([A]^\varphi_\alpha\) is convex.
\end{proof}

Observable convexity is a particular case of preorder-based convexity. Indeed, every observable \(\varphi:L\to[0,1]\) induces a total preorder on \(L\), defined by
\[
a\preccurlyeq_{\varphi} b \Longleftrightarrow \varphi(a)\le \varphi(b).
\]
Since \([0,1]\) is totally ordered, \(\preccurlyeq_{\varphi}\) is always a total preorder.

\begin{theorem}\label{teo_observable_induced_order}
Let \((X,\Gamma)\) be a point-convex segment-generated abstract convexity space, let \(L\) be a nonempty set, let \(\varphi:L\to[0,1]\), and let \(A:X\to L\). Then the following statements are equivalent:
\begin{enumerate}
    \item \(A\) is \(\varphi\)-convex;
    \item \(A\) is \(\preccurlyeq_{\varphi}\)-convex.
\end{enumerate}
\end{theorem}

\begin{proof}
By Theorem~\ref{teo_obs_local}, \(A\) is \(\varphi\)-convex if and only if, for every \(x,z\in X\) and every \(y\in[x,z]_{\Gamma}\),
\[
\varphi(A(y))\ge \min\{\varphi(A(x)),\varphi(A(z))\}.
\]
Since \([0,1]\) is totally ordered, this is equivalent to
\[
\varphi(A(x))\le \varphi(A(y))
\quad\text{or}\quad
\varphi(A(z))\le \varphi(A(y)).
\]
By the definition of \(\preccurlyeq_{\varphi}\), this means
\[
A(x)\preccurlyeq_{\varphi}A(y)
\quad\text{or}\quad
A(z)\preccurlyeq_{\varphi}A(y).
\]
By Theorem~\ref{teo_orden_local}, this is equivalent to \(A\) being \(\preccurlyeq_{\varphi}\)-convex.
\end{proof}

\smallskip

\begin{remark}
The aggregation-based convexity condition recalled in \eqref{eq:aggregation-based-thfs-convexity} is recovered from Definition~\ref{def_observable_convex} by taking \(L\) as the family of typical hesitant fuzzy values and choosing the aggregation function as the observable \(\varphi\), in the case where the domain convexity is induced by a linear order.
\end{remark}

\smallskip

The next example visualizes why observable convexity should not be confused with intrinsic lattice convexity.

\begin{example}\label{ex3}
Let \(L=\{0<u<1\}\) be the three-element chain, and define an observable \(\varphi:L\to[0,1]\) by
\[
\varphi(0)=0,\qquad \varphi(u)=1,\qquad \varphi(1)=1.
\]
Let \(X=\{x_1<x_2<x_3\}\), endowed with its usual order convexity, and define \(A:X\to L\) by
\[
A(x_1)=1,\qquad A(x_2)=u,\qquad A(x_3)=1.
\]
Then \(A\) is \(\varphi\)-convex, since
\[
\varphi(A(x_1))=\varphi(A(x_2))=\varphi(A(x_3))=1.
\]
However, \(A\) is not intrinsically convex, because \(x_2\in[x_1,x_3]_{\Gamma}\),
\[
A(x_1)\wedge A(x_3)=1\wedge 1=1,
\]
whereas \(A(x_2)=u<1\). Thus, the observable hides the loss of lattice information at the intermediate point. Figure~\ref{fig:observable-not-intrinsic-example} visualizes this example.

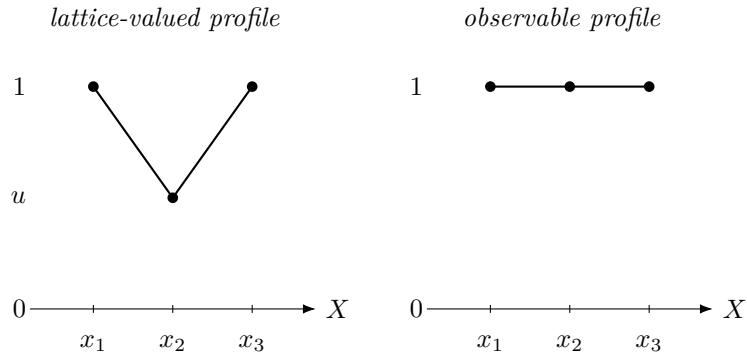
\begin{figure}[htbp]
\centering
\begin{tikzpicture}[
    x=1.05cm,
    y=1.05cm,
    >=Latex,
    font=\small,
    point/.style={circle,fill=black,inner sep=2pt}
]

\node at (1.45,3.65) {\emph{lattice-valued profile}};

\draw[->] (-0.25,0) -- (3.35,0) node[right] {$X$};

\node[left=5pt] at (0,0) {$0$};
\node[left=5pt] at (0,1.4) {$u$};
\node[left=5pt] at (0,2.8) {$1$};

\foreach \x/\lab in {0.55/x_1,1.55/x_2,2.55/x_3}{
  \draw (\x,-0.05) -- (\x,0.05);
  \node[below=6pt] at (\x,0) {$\lab$};
}

\coordinate (A1) at (0.55,2.8);
\coordinate (A2) at (1.55,1.4);
\coordinate (A3) at (2.55,2.8);

\draw[thick] (A1)--(A2)--(A3);
\foreach \P in {A1,A2,A3}{
  \fill (\P) circle (2pt);
}

\begin{scope}[xshift=5.25cm]

\node at (1.45,3.65) {\emph{observable profile}};

\draw[->] (-0.25,0) -- (3.35,0) node[right] {$X$};

\node[left=5pt] at (0,0) {$0$};
\node[left=5pt] at (0,2.8) {$1$};

\foreach \x/\lab in {0.55/x_1,1.55/x_2,2.55/x_3}{
  \draw (\x,-0.05) -- (\x,0.05);
  \node[below=6pt] at (\x,0) {$\lab$};
}

\coordinate (B1) at (0.55,2.8);
\coordinate (B2) at (1.55,2.8);
\coordinate (B3) at (2.55,2.8);

\draw[thick] (B1)--(B2)--(B3);
\foreach \P in {B1,B2,B3}{
  \fill (\P) circle (2pt);
}

\end{scope}

\end{tikzpicture}
\caption{Observable convexity may hold even when intrinsic convexity fails.}
\label{fig:observable-not-intrinsic-example}
\end{figure}
\end{example}

In the next section, we study when such scalar descriptions preserve or reconstruct the intrinsic one.

\smallskip

Figure~\ref{fig_convexity1} summarizes the relationships between the two schemes introduced in this section.

\begin{figure}[htbp]
\centering
\begin{tikzpicture}[
    scale=0.72,
    transform shape,
    >=Latex,
    every node/.style={align=center},
    box/.style={
        draw,
        rounded corners,
        thick,
        minimum width=40mm,
        minimum height=10mm,
        inner sep=4pt
    },
    litbox/.style={
        draw,
        rounded corners,
        thin,
        fill=gray!8,
        minimum width=54mm,
        inner sep=3pt,
        font=\footnotesize
    },
    title/.style={
        font=\bfseries
    },
    bridge/.style={
        dashed,
        thick
    },
    group/.style={
        draw,
        rounded corners,
        dashed,
        inner sep=8pt,
        thick
    }
]

\node[title] (t1) at (0,2.7) {Operational perspective};
\node[title] (t2) at (10.7,2.7) {Relational perspective};

\node[box] (lotimes) at (0,1.0)
{\((L,\otimes)\)-convexity\\[1mm]
\(A(y)\geq A(x)\otimes A(z)\)};

\node[box] (intrinsic) at (0,-1.1)
{Intrinsic convexity\\[1mm]
\(A(y)\geq A(x)\wedge A(z)\)};

\draw[->, thick] (lotimes) -- (intrinsic)
node[midway, right=1mm] {\footnotesize \(\otimes=\wedge\)};

\node[litbox] (litleft) at (0,-3.85)
{Min-based fuzzy convexity \citep{Zadeh1965}\\[1mm]
\(F\)-convexity \citep{Diaz2017}\\[1mm]
Lattice-/poset-valued cut-based viewpoints\\
\citep{Maruyama2009,JanisSeseljaTepavcevic2017}};

\node[group, fit=(lotimes)(intrinsic)(litleft)] (internalgroup) {};

\node[box] (preorder) at (10.7,1.0)
{Preorder-based convexity\\[1mm]
\(A(x)\preccurlyeq A(y)\) or \(A(z)\preccurlyeq A(y)\)};

\node[box] (latticeorder) at (7.7,-1.1)
{Lattice-order convexity\\[1mm]
\(\preccurlyeq=\leq\)};

\node[box] (observable) at (13.7,-1.1)
{Observable convexity\\[1mm]
\(\varphi(A(y))\ge \min\{\varphi(A(x)),\varphi(A(z))\}\)};

\draw[->, thick] (preorder.south west) -- (latticeorder.north)
node[pos=0.45, above left=1mm] {\footnotesize lattice order};

\draw[->, thick] (preorder.south east) -- (observable.north)
node[pos=0.45, above right=1mm] {\footnotesize induced by\\[-1mm] \(\varphi:L\to[0,1]\)};

\node[litbox] (litright) at (10.7,-4.05)
{Score-based convexity for THFSs \citep{RashidBeg2016}\\[1mm]
Aggregation-based convexity \citep{Huidobro2021}\\[1mm]
Order-based convexity for IVFSs and THFSs\\
\citep{Huidobro2022,Huidobro2025}};

\node[group, fit=(preorder)(latticeorder)(observable)(litright)] (externalgroup) {};

\node[group, fit=(latticeorder)(observable), inner sep=6pt] (relsubcases) {};

\draw[bridge] (intrinsic.east) -- 
node[midway, above=1.2mm, font=\footnotesize]
{}
(relsubcases.west);

\end{tikzpicture}
\caption{Operational and relational schemes for convexity on \(L\)-fuzzy sets.}
\label{fig_convexity1}
\end{figure}
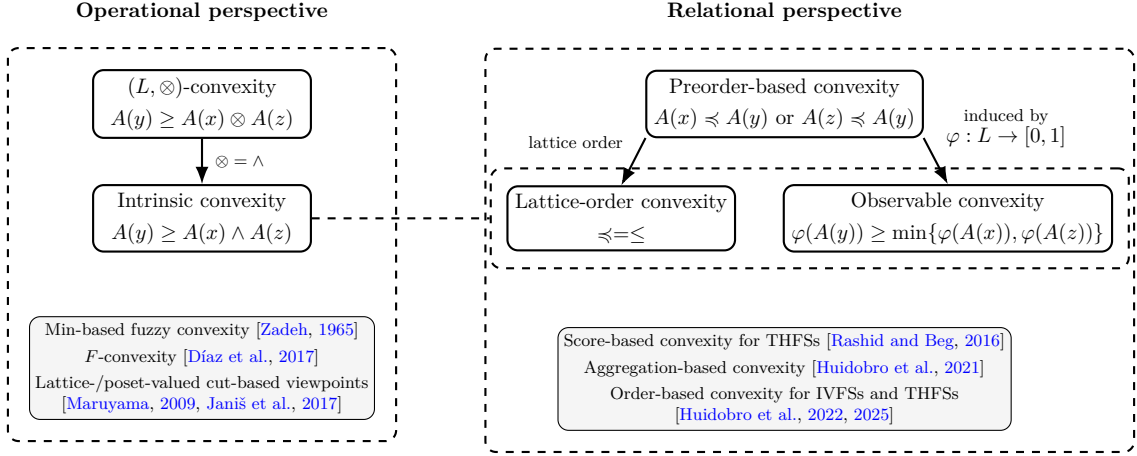


\section{Bridges between intrinsic and scalarized convexity}\label{section:bridges}

In the previous section we introduced two general convexity models on \(L\)-fuzzy sets. The operational perspective contains intrinsic lattice convexity as the distinguished case, while the relational perspective contains preorder-based and observable convexities. The purpose of this section is to compare these notions and to identify when a scalar or relational description is faithful to the intrinsic lattice structure.

\smallskip

In the converse statements below, we shall use the following non-degeneracy condition. We say that \((X,\Gamma)\) has a \emph{nontrivial segment} if there exist \(p,r\in X\) and \(q\in[p,r]_{\Gamma}\setminus\{p,r\}\). This assumption only excludes degenerate domains with no genuine intermediate points. This condition is used to obtain necessity in the characterization results.

\subsection{Intrinsic convexity versus observable convexity}

\smallskip

We first study when intrinsically convex mappings remain convex after applying a scalar observable. The following result shows that this is possible only when the observable transforms lattice meets into minima.

\begin{theorem}\label{teo_sharp_observable}
Let \((X,\Gamma)\) be a point-convex segment-generated abstract convexity space with a nontrivial segment, let \((L,\leq,\wedge,\vee)\) be a lattice, and let \(\varphi:L\to[0,1]\) be an observable. Then the following statements are equivalent:
\begin{enumerate}
    \item every intrinsically convex \(L\)-fuzzy set \(A:X\to L\) is \(\varphi\)-convex;
    \item \(\varphi(a\wedge b)=\min\{\varphi(a),\varphi(b)\}\) for all \(a,b\in L\).
\end{enumerate}
\end{theorem}

\begin{proof}
Assume first (2). Observe that condition (2) implies that \(\varphi\) is monotone. Indeed, if \(a\leq b\), then \(a=a\wedge b\), and hence
\begin{equation*}
\varphi(a)
=
\varphi(a\wedge b)
=
\min\{\varphi(a),\varphi(b)\},
\end{equation*}
which gives \(\varphi(a)\leq\varphi(b)\).

Let \(A:X\to L\) be intrinsically convex, let \(x,z\in X\), and let \(y\in[x,z]_{\Gamma}\). Then \(A(y)\ge A(x)\wedge A(z)\). By the monotonicity of \(\varphi\),
\begin{equation*}
\varphi(A(y))
\ge
\varphi(A(x)\wedge A(z))
=
\min\{\varphi(A(x)),\varphi(A(z))\}.
\end{equation*}
Hence \(A\) is \(\varphi\)-convex by Theorem~\ref{teo_obs_local}.

Conversely, assume (1). We first show that \(\varphi\) is monotone. Let \(a,b\in L\) with \(a\leq b\). Choose \(p,r\in X\) and \(q\in[p,r]_{\Gamma}\setminus\{p,r\}\), and define \(B:X\to L\) by
\begin{equation*}
B(q)=b,\qquad B(x)=a\quad\text{for every }x\in X\setminus\{q\}.
\end{equation*}
We claim that \(B\) is intrinsically convex. Let \(u,v\in X\) and \(w\in[u,v]_{\Gamma}\). If \(B(u)=B(v)=b\), then \(u=v=q\), and point-convexity gives \(w=q\), so the inequality is immediate. Otherwise, at least one of \(B(u)\) and \(B(v)\) is equal to \(a\), and therefore
\begin{equation*}
B(u)\wedge B(v)\leq a\leq B(w).
\end{equation*}
Hence \(B\) is intrinsically convex. By (1), \(B\) is \(\varphi\)-convex. Applying Theorem~\ref{teo_obs_local} to \(p,r\) and \(q\in[p,r]_{\Gamma}\), we obtain
\begin{equation*}
\varphi(b)
=
\varphi(B(q))
\geq
\min\{\varphi(B(p)),\varphi(B(r))\}
=
\varphi(a).
\end{equation*}
Thus \(\varphi\) is monotone.

Now fix \(a,b\in L\), set \(m:=a\wedge b\), and define \(A:X\to L\) by
\begin{equation*}
A(p)=a,\qquad A(r)=b,\qquad A(x)=m\quad\text{for every }x\in X\setminus\{p,r\}.
\end{equation*}
We claim that \(A\) is intrinsically convex. Let \(u,v\in X\) and \(w\in[u,v]_{\Gamma}\). If \(u=v=p\) or \(u=v=r\), then point-convexity gives \(w=u\), and the inequality is immediate. If \(\{u,v\}=\{p,r\}\), then \(A(u)\wedge A(v)=m\le A(w)\). In all remaining cases, at least one of \(A(u)\) and \(A(v)\) is equal to \(m\), so \(A(u)\wedge A(v)\le m\le A(w)\). Hence \(A\) is intrinsically convex.

By (1), \(A\) is \(\varphi\)-convex. Applying Theorem~\ref{teo_obs_local} to \(p,r\) and \(q\in[p,r]_{\Gamma}\), we obtain
\begin{equation*}
\varphi(a\wedge b)=\varphi(A(q))
\ge
\min\{\varphi(A(p)),\varphi(A(r))\}
=
\min\{\varphi(a),\varphi(b)\}.
\end{equation*}
On the other hand, since \(a\wedge b\le a\) and \(a\wedge b\le b\), the monotonicity of \(\varphi\) gives
\begin{equation*}
\varphi(a\wedge b)\le \varphi(a),
\qquad
\varphi(a\wedge b)\le \varphi(b).
\end{equation*}
Therefore \(\varphi(a\wedge b)\le \min\{\varphi(a),\varphi(b)\}\), and equality follows.
\end{proof}

Thus, scalarization preserves intrinsic convexity only under a meet-preservation condition. This result is particularly important because it already excludes many numerical summaries from being faithful to intrinsic convexity.

\subsubsection{Reconstruction by one observable}

\smallskip

We must note that preservation is not the same as reconstruction. We now characterize when a single observable not only preserves intrinsic convexity, but completely recovers it.

\begin{theorem}\label{teo_single_observable_exact}
Let \((X,\Gamma)\) be a point-convex segment-generated abstract convexity space with a nontrivial segment, let \((L,\leq,\wedge,\vee)\) be a lattice, and let \(\varphi:L\to[0,1]\) be an observable. Then the following statements are equivalent:
\begin{enumerate}
    \item for every \(L\)-fuzzy set \(A:X\to L\), \(A\) is intrinsically convex if and only if \(A\) is \(\varphi\)-convex;
    \item \(\varphi(a\wedge b)=\min\{\varphi(a),\varphi(b)\}\) for all \(a,b\in L\), and \(\varphi(a)\le \varphi(b)\) implies \(a\le b\) for all \(a,b\in L\), that is, \(\varphi\) reflects the order.
\end{enumerate}
\end{theorem}

\begin{proof}
Assume (2). By the meet-minimum identity, \(\varphi\) is monotone. If \(A\) is intrinsically convex, then \(A\) is \(\varphi\)-convex by Theorem~\ref{teo_sharp_observable}. Conversely, assume that \(A\) is \(\varphi\)-convex. Let \(x,z\in X\) and let \(y\in[x,z]_{\Gamma}\). By Theorem~\ref{teo_obs_local},
\begin{equation*}
\varphi(A(y))
\ge
\min\{\varphi(A(x)),\varphi(A(z))\}
=
\varphi(A(x)\wedge A(z)).
\end{equation*}
Since \(\varphi\) reflects the order, it follows that \(A(x)\wedge A(z)\le A(y)\). Hence \(A\) is intrinsically convex.

Assume now (1). By Theorem~\ref{teo_sharp_observable}, one has
\begin{equation*}
\varphi(a\wedge b)=\min\{\varphi(a),\varphi(b)\}
\end{equation*}
for all \(a,b\in L\). It remains to prove that \(\varphi\) reflects the order. Let \(a,b\in L\) satisfy \(\varphi(a)\le \varphi(b)\). Choose \(p,r\in X\) and \(q\in[p,r]_{\Gamma}\setminus\{p,r\}\), and define \(A:X\to L\) by
\begin{equation*}
A(q)=b,\qquad A(x)=a\quad\text{for every }x\in X\setminus\{q\}.
\end{equation*}

We claim that \(A\) is \(\varphi\)-convex. By Theorem~\ref{teo_obs_local}, it is enough to check that
\begin{equation*}
\varphi(A(y))\ge \min\{\varphi(A(x)),\varphi(A(z))\}
\end{equation*}
for all \(x,z\in X\) and \(y\in[x,z]_{\Gamma}\). If \(A(x)=A(z)=b\), then \(x=z=q\), and point-convexity gives \(y=q\), so the inequality is immediate. Otherwise, at least one of \(A(x)\) and \(A(z)\) is equal to \(a\), and therefore
\begin{equation*}
\min\{\varphi(A(x)),\varphi(A(z))\}\le \varphi(a).
\end{equation*}
Since \(A(y)\in\{a,b\}\) and \(\varphi(a)\le \varphi(b)\), one has \(\varphi(A(y))\ge \varphi(a)\). Hence \(A\) is \(\varphi\)-convex.

By (1), \(A\) is intrinsically convex. Applying intrinsic convexity to \(p,r\) and \(q\in[p,r]_{\Gamma}\), we get
\begin{equation*}
a=A(p)\wedge A(r)\le A(q)=b.
\end{equation*}
Thus \(\varphi(a)\le\varphi(b)\) implies \(a\le b\), so \(\varphi\) reflects the order.
\end{proof}

Theorem~\ref{teo_single_observable_exact} gives an obstruction to scalar reconstruction. A single observable must preserve meets as minima and must also reflect the underlying lattice order. Since \([0,1]\) is totally ordered, the existence of an order-reflecting map \(\varphi:L\to[0,1]\) forces \(L\) itself to be totally ordered. Hence this reconstruction mechanism is unavailable in non-chain lattices.

\begin{example}\label{ex4}
Let \(L=\{0,p,q,1\}\) be the four-element diamond lattice, where \(0<p<1\), \(0<q<1\), and \(p\) and \(q\) are incomparable \citep{Gratzer2011}. No map \(\varphi:L\to[0,1]\) can reflect the order of \(L\). Indeed, since \([0,1]\) is totally ordered, either \(\varphi(p)\le\varphi(q)\) or \(\varphi(q)\le\varphi(p)\). If \(\varphi\) reflected the order, this would imply either \(p\le q\) or \(q\le p\), which is impossible. Hence no single observable can reconstruct intrinsic convexity on this lattice.
\end{example}

The implication from the second statement in Theorem~\ref{teo_single_observable_exact} to the first one remains valid without assuming the existence of a nontrivial segment.

\begin{corollary}\label{corolario_single_observable}
Let \((X,\Gamma)\) be a point-convex segment-generated abstract convexity space, let \((L,\leq,\wedge,\vee)\) be a lattice, and let \(\varphi:L\to[0,1]\) be an observable. Assume that \(\varphi(a\wedge b)=\min\{\varphi(a),\varphi(b)\}\) for all \(a,b\in L\), and that \(\varphi\) reflects the order. Then, for every \(L\)-fuzzy set \(A:X\to L\), the following statements are equivalent:
\begin{enumerate}
    \item \(A\) is intrinsically convex;
    \item \(A\) is \(\varphi\)-convex.
\end{enumerate}
\end{corollary}

\begin{proof}
The implication from intrinsic convexity to \(\varphi\)-convexity follows as in the proof of Theorem~\ref{teo_sharp_observable}. Conversely, if \(A\) is \(\varphi\)-convex, then the argument in the first part of the proof of Theorem~\ref{teo_single_observable_exact} gives \(A(x)\wedge A(z)\le A(y)\) for all \(x,z\in X\) and \(y\in[x,z]_{\Gamma}\). Hence \(A\) is intrinsically convex.
\end{proof}

Note that, by contrast, the converse implication in Theorem~\ref{teo_single_observable_exact} requires the existence of a nontrivial segment. This assumption is necessary to derive the meet-minimum identity and the order-reflection property from exact scalar reconstruction.


\subsubsection{Reconstruction by families of observables}

\smallskip

When a single observable is not enough, one may consider a family of observables. In this case, we demonstrate that the relevant condition for the family is that it separates the order.

\begin{definition}\label{def_separating}
Let \((L,\leq)\) be a poset. A family of maps \(\Phi=\{\varphi_i:L\to[0,1]\}_{i\in I}\) is said to be \emph{separating} if, for all \(a,b\in L\),
\begin{equation*}
a\le b
\quad\Longleftrightarrow\quad
\varphi_i(a)\le \varphi_i(b)\quad\text{for every }i\in I.
\end{equation*}
\end{definition}

\begin{theorem}\label{teo_exact_family}
Let \((X,\Gamma)\) be a point-convex segment-generated abstract convexity space with a nontrivial segment, let \((L,\leq,\wedge,\vee)\) be a lattice, and let \(\Phi=\{\varphi_i:L\to[0,1]\}_{i\in I}\) be a family of observables. Then the following statements are equivalent:
\begin{enumerate}
    \item for every \(L\)-fuzzy set \(A:X\to L\), \(A\) is intrinsically convex if and only if \(A\) is \(\varphi_i\)-convex for every \(i\in I\);
    \item \(\Phi\) is separating, that is, \(a\leq b\) if and only if \(\varphi_i(a)\leq\varphi_i(b)\) for every \(i\in I\), and each \(\varphi_i\) preserves finite meets as minima, namely
    \[
    \varphi_i(a\wedge b)=\min\{\varphi_i(a),\varphi_i(b)\}
    \]
    for all \(a,b\in L\) and every \(i\in I\).
\end{enumerate}
\end{theorem}

\begin{proof}
Assume first (2). If \(A\) is intrinsically convex, then, for every \(i\in I\), Theorem~\ref{teo_sharp_observable} implies that \(A\) is \(\varphi_i\)-convex. Conversely, assume that \(A\) is \(\varphi_i\)-convex for every \(i\in I\). Given \(x,z\in X\) and \(y\in[x,z]_{\Gamma}\), Theorem~\ref{teo_obs_local} gives \(\varphi_i(A(y))\geq \min\{\varphi_i(A(x)),\varphi_i(A(z))\}=\varphi_i(A(x)\wedge A(z))\) for every \(i\in I\). Since \(\Phi\) is separating, \(A(x)\wedge A(z)\leq A(y)\). Hence \(A\) is intrinsically convex.

Conversely, assume (1). Since every intrinsically convex \(L\)-fuzzy set is \(\varphi_i\)-convex for every \(i\in I\), Theorem~\ref{teo_sharp_observable} gives the meet-minimum identity for each \(\varphi_i\). In particular, every \(\varphi_i\) is monotone. It remains to show that \(\Phi\) is separating. The implication \(a\leq b\Rightarrow \varphi_i(a)\leq\varphi_i(b)\) for every \(i\in I\) now follows from the monotonicity of each \(\varphi_i\). For the converse, suppose that \(\varphi_i(a)\leq\varphi_i(b)\) for every \(i\in I\). Choose \(p,r\in X\) and \(q\in[p,r]_{\Gamma}\setminus\{p,r\}\), and define \(A:X\to L\) by \(A(q)=b\) and \(A(x)=a\) for every \(x\neq q\). For each \(i\in I\), the map \(A\) is \(\varphi_i\)-convex: if both endpoints have value \(b\), then they are both equal to \(q\) and point-convexity gives only the trivial segment; otherwise at least one endpoint has value \(a\), and the inequality follows from \(\varphi_i(a)\leq\varphi_i(b)\). By (1), \(A\) is intrinsically convex. Applying intrinsic convexity to \(p,r\) and \(q\in[p,r]_{\Gamma}\), we obtain \(a=A(p)\wedge A(r)\leq A(q)=b\). Thus \(\Phi\) is separating.
\end{proof}

Observe that Theorem~\ref{teo_exact_family} provides the bridge between observable convexity and order representation. It shows that reconstructing intrinsic convexity by a family of scalar observables necessarily requires the same family to separate the codomain order. Hence, any obstruction to a finite scalar representation of the codomain order becomes an obstruction to a finite scalar reconstruction of intrinsic convexity.

\begin{corollary}\label{cor_dimension_obstruction}
Let \((X,\Gamma)\) be a point-convex segment-generated abstract convexity space with a nontrivial segment, and let \((L,\leq,\wedge,\vee)\) be a lattice. Assume that, for every \(m\in\mathbb N\), the poset \(L\) contains a subposet of order dimension greater than \(m\). Then no finite family of observables from \(L\) to \([0,1]\) can reconstruct intrinsic convexity on all \(L\)-fuzzy sets.
\end{corollary}

\begin{proof}
If such a finite family reconstructed intrinsic convexity on all \(L\)-fuzzy sets, then Theorem~\ref{teo_exact_family} would imply that it separates the order of \(L\). Hence \(L\) would order-embed into a product of \(m\) chains through \(a\mapsto(\varphi_1(a),\ldots,\varphi_m(a))\). Therefore every subposet of \(L\) would have order dimension at most \(m\), contradicting the hypothesis.
\end{proof}

As shown in Section~\ref{sec:hesitant}, Corollary~\ref{cor_dimension_obstruction} provides the abstract criterion behind the main non-representability result of the article. Once the relevant codomain lattice contains finite subposets of arbitrarily large order dimension, no finite family of scalar observables can reconstruct the associated intrinsic convexity.


\smallskip

Despite the previous limitation, for finite domains, the convexity of a fixed \(L\)-fuzzy set admits a finite certificate. However, this does not represent the whole codomain order; it only provides a finite test for one profile.

\begin{theorem}\label{teo_finite}
Let \(X=\{x_1,\ldots,x_n\}\) be a finite set endowed with a point-convex segment-generated abstract convexity \(\Gamma\), let \((L,\leq,\wedge,\vee)\) be a lattice, and let \(A:X\to L\). Define
\begin{equation*}
M_A:=\{A(x_i)\wedge A(x_j): 1\le i\le j\le n\}.
\end{equation*}
Then the following statements are equivalent:
\begin{enumerate}
    \item \(A\) is intrinsically convex;
    \item for every \(a\in M_A\), the principal cut \([A]_a=\{x\in X:\ a\le A(x)\}\) is convex;
    \item for every \(a\in M_A\), the mapping \(A\) is \(\eta_a\)-convex, where \(\eta_a(u)=1\) if \(a\le u\) and \(\eta_a(u)=0\) otherwise.
\end{enumerate}
In particular, intrinsic convexity can be verified using at most \(n(n+1)/2\) levels.
\end{theorem}

\begin{proof}
The implication (1)\(\Rightarrow\)(2) follows from Theorem~\ref{teo_princ_cuts}. Assume (2). Let \(x_i,x_j\in X\), let \(y\in[x_i,x_j]_{\Gamma}\), and set
\begin{equation*}
a:=A(x_i)\wedge A(x_j).
\end{equation*}
Then \(a\in M_A\), and \(a\le A(x_i)\), \(a\le A(x_j)\), so \(x_i,x_j\in[A]_a\). By convexity of \([A]_a\), we get \([x_i,x_j]_{\Gamma}\subseteq[A]_a\), and therefore \(y\in[A]_a\). Thus
\begin{equation*}
A(y)\ge a=A(x_i)\wedge A(x_j).
\end{equation*}
Hence \(A\) is intrinsically convex.

Finally, fix \(a\in M_A\). The map \(\eta_a\) only takes the values \(0\) and \(1\). Hence \([A]^{\eta_a}_0=X\), while \([A]^{\eta_a}_\alpha=[A]_a\) for every \(\alpha\in(0,1]\). Therefore \(A\) is \(\eta_a\)-convex if and only if \([A]_a\) is convex. This proves the equivalence between (2) and (3).

The set \(M_A\) contains at most one element for each unordered pair \(\{i,j\}\), with \(1\le i\le j\le n\). Since the number of such pairs is \(n(n+1)/2\), at most \(n(n+1)/2\) levels are needed.
\end{proof}

\subsection{Intrinsic convexity versus lattice-order convexity}

\smallskip

We now compare intrinsic convexity with the particular relational notion obtained from the lattice order itself. This comparison clarifies why the intrinsic and relational viewpoints coincide in chains, but not in general lattices.

\begin{theorem}\label{teo_order_intrinsic}
Let \((L,\leq,\wedge,\vee)\) be a lattice, and let \(\leq\) denote its lattice order. Then:
\begin{enumerate}
    \item for every point-convex segment-generated abstract convexity space \((X,\Gamma)\), every \(\leq\)-convex \(L\)-fuzzy set \(A:X\to L\) is intrinsically convex;
    \item if \(L\) is totally ordered, then for every point-convex segment-generated abstract convexity space \((X,\Gamma)\) and every \(L\)-fuzzy set \(A:X\to L\), \(A\) is intrinsically convex if and only if \(A\) is \(\leq\)-convex;
    \item conversely, if \((X,\Gamma)\) has a nontrivial segment and, for every \(L\)-fuzzy set \(A:X\to L\), \(A\) is intrinsically convex if and only if \(A\) is \(\leq\)-convex, then \(L\) is totally ordered.
\end{enumerate}
\end{theorem}

\begin{proof}
For (1), let \(A:X\to L\) be \(\leq\)-convex, let \(x,z\in X\), and let \(y\in[x,z]_{\Gamma}\). By Theorem~\ref{teo_orden_local}, either \(A(x)\le A(y)\) or \(A(z)\le A(y)\). In either case, \(A(x)\wedge A(z)\le A(y)\). Hence \(A\) is intrinsically convex.

For (2), assume that \(L\) is totally ordered. Then \(A(x)\wedge A(z)=\min_{\le}\{A(x),A(z)\}\), so \(A(y)\ge A(x)\wedge A(z)\) holds if and only if \(A(x)\le A(y)\) or \(A(z)\le A(y)\). By Theorem~\ref{teo_orden_local}, this is exactly \(\leq\)-convexity.

For (3), assume that intrinsic convexity and \(\leq\)-convexity coincide for all \(A:X\to L\), and suppose that \(L\) is not totally ordered. Then there exist incomparable elements \(a,b\in L\). Choose \(p,r\in X\) and \(q\in[p,r]_{\Gamma}\setminus\{p,r\}\). Set \(m:=a\wedge b\), and define \(A:X\to L\) by
\begin{equation*}
A(p)=a,\qquad A(r)=b,\qquad A(x)=m\quad\text{for every }x\in X\setminus\{p,r\}.
\end{equation*}
As in the proof of Theorem~\ref{teo_sharp_observable}, \(A\) is intrinsically convex. However, \(A\) is not \(\leq\)-convex, because applying Theorem~\ref{teo_orden_local} to \(p,r\) and \(q\in[p,r]_{\Gamma}\) would require
\begin{equation*}
a\le a\wedge b
\quad\text{or}\quad
b\le a\wedge b,
\end{equation*}
which is impossible for incomparable \(a,b\). Therefore \(L\) must be totally ordered.
\end{proof}

Thus, intrinsic convexity and lattice-order convexity coincide for every profile whenever the codomain lattice is totally ordered. Conversely, on domains with a nontrivial segment, this coincidence for all \(L\)-fuzzy sets characterizes totality of the codomain order. This explains why scalar or order-based formulations behave well in chain-like settings, but not necessarily in genuinely lattice-valued ones.

\begin{corollary}\label{cororlario_lattice_order_total}
Let \((X,\Gamma)\) be a point-convex segment-generated abstract convexity space, and let \((L,\leq,\wedge,\vee)\) be a totally ordered lattice. Then, for every \(L\)-fuzzy set \(A:X\to L\), the following statements are equivalent:
\begin{enumerate}
    \item \(A\) is intrinsically convex;
    \item \(A\) is \(\leq\)-convex.
\end{enumerate}
\end{corollary}

\begin{proof}
This follows directly from Theorem~\ref{teo_order_intrinsic}.
\end{proof}

Combining this with Corollary~\ref{corolario_single_observable}, we obtain the scalar case in totally ordered codomains.

\begin{corollary}\label{corolario_single_score}
Let \((X,\Gamma)\) be a point-convex segment-generated abstract convexity space, let \((L,\leq,\wedge,\vee)\) be a totally ordered lattice, and let \(\varphi:L\to[0,1]\) satisfy
\begin{equation*}
a\le b
\quad\Longleftrightarrow\quad
\varphi(a)\le \varphi(b)
\end{equation*}
for all \(a,b\in L\). Then, for every \(L\)-fuzzy set \(A:X\to L\), the following statements are equivalent:
\begin{enumerate}
    \item \(A\) is intrinsically convex;
    \item \(A\) is \(\leq\)-convex;
    \item \(A\) is \(\varphi\)-convex.
\end{enumerate}
\end{corollary}

\begin{proof}
Since \(L\) is totally ordered and \(\varphi\) preserves and reflects the order, one has
\begin{equation*}
\varphi(a\wedge b)=\min\{\varphi(a),\varphi(b)\}
\end{equation*}
for all \(a,b\in L\). The equivalence between (1) and (2) follows from Corollary~\ref{cororlario_lattice_order_total}, and the equivalence between (1) and (3) follows from Corollary~\ref{corolario_single_observable}.
\end{proof}

The following example illustrates the scalar reconstruction obtained in Corollary~\ref{corolario_single_score} for a totally ordered codomain.

\begin{example}\label{ex5}
Let \(L=\{0<u<v<1\}\) be the four-element chain, endowed with its usual lattice structure, and define an observable \(\varphi:L\to[0,1]\) by
\begin{equation*}
\varphi(0)=0,\qquad \varphi(u)=0.4,\qquad \varphi(v)=0.7,\qquad \varphi(1)=1.
\end{equation*}
Then \(\varphi\) preserves and reflects the order. Since \(L\) is totally ordered, Corollary~\ref{corolario_single_score} applies.

Let \(X=\{x_1<x_2<x_3<x_4<x_5\}\), endowed with its usual order convexity, and define \(A:X\to L\) by
\begin{equation*}
A(x_1)=0,\qquad A(x_2)=u,\qquad A(x_3)=v,\qquad A(x_4)=u,\qquad A(x_5)=0.
\end{equation*}
Then \(A\) is intrinsically convex, equivalently \(\leq\)-convex and \(\varphi\)-convex. Indeed, its principal cuts are convex; for instance,
\begin{equation*}
[A]_u=\{x_2,x_3,x_4\},
\qquad
[A]_v=\{x_3\}.
\end{equation*}
The observable profile is
\begin{equation*}
0,\quad 0.4,\quad 0.7,\quad 0.4,\quad 0,
\end{equation*}
which satisfies the usual min-based convexity condition.
\end{example}


\section{Pointwise operators and preservation of intrinsic convexity}\label{sec:pointwise}

The previous section identified the order-theoretic conditions under which scalar observables preserve or reconstruct intrinsic convexity. We now address the complementary preservation problem of whether intrinsic convexity is stable under pointwise operators acting directly on the lattice-valued codomain. This question is relevant since, in several fuzzy contexts, profiles are often combined through aggregation, conjunction, disjunction, or other codomain operations. Related preservation problems for aggregation operators and fuzzy quasiconvexity were previously studied by \citet{JanisKralRencova2013}. Following the approach to induced operators on lattices developed in \citet{LobilloMerinoNavarroSantos2021,MerinoNavarroSantos2022}, we characterize when pointwise combinations preserve intrinsic convexity in the present lattice-valued framework. The result below shows that preservation is governed by compatibility with the meet operation of the codomain lattice. In fact, this compatibility already forces monotonicity of the operator, while monotonicity alone is not sufficient. In any case, no completeness or distributivity assumption on \(L\) is required.

\begin{definition}\label{def_induced_operator}
Let \(m\geq 2\), let \(X\) be a nonempty set, let \((L,\leq,\wedge,\vee)\) be a lattice, and let \(T:L^m\to L\) be an \(m\)-ary operator. The \emph{operator induced by \(T\)} on \(L^X\) is the map
\[
T_*:\big(L{^X}\big)^{m} \to L^X
\]
defined by
\[
T_*(A_1,\dots,A_m)(x):=T(A_1(x),\dots,A_m(x)),
\qquad x\in X.
\]
\end{definition}

\begin{theorem}\label{teo_sharp_induced}
Let \(m\geq 2\), let \((X,\Gamma)\) be a point-convex segment-generated abstract convexity space with a nontrivial segment, let \((L,\leq,\wedge,\vee)\) be a lattice, and let \(T:L^m\to L\) be an \(m\)-ary operator. Then the following statements are equivalent:
\begin{enumerate}
    \item for every family of intrinsically convex \(L\)-fuzzy sets \(A_1,\dots,A_m:X\to L\), the induced \(L\)-fuzzy set \(T_*(A_1,\dots,A_m):X\to L\) is intrinsically convex;
    \item for all \(a_1,\dots,a_m,b_1,\dots,b_m\in L\), one has
    \[
    T(a_1\wedge b_1,\dots,a_m\wedge b_m)
    =
    T(a_1,\dots,a_m)\wedge T(b_1,\dots,b_m).
    \]
\end{enumerate}
\end{theorem}

\begin{proof}
Assume first (2). We first observe that \(T\) is monotone with respect to the product order on \(L^m\). Indeed, if \(c_s\leq d_s\) for every \(s=1,\dots,m\), then \(c_s=c_s\wedge d_s\), and therefore
\[
T(c_1,\dots,c_m)
=
T(c_1\wedge d_1,\dots,c_m\wedge d_m)
=
T(c_1,\dots,c_m)\wedge T(d_1,\dots,d_m).
\]
Hence \(T(c_1,\dots,c_m)\leq T(d_1,\dots,d_m)\).

Let \(A_1,\dots,A_m:X\to L\) be intrinsically convex, let \(x,z\in X\), and let \(y\in[x,z]_{\Gamma}\). For each \(s=1,\dots,m\), intrinsic convexity gives \(A_s(y)\geq A_s(x)\wedge A_s(z)\). By the monotonicity of \(T\),
\[
T(A_1(y),\dots,A_m(y))
\geq
T(A_1(x)\wedge A_1(z),\dots,A_m(x)\wedge A_m(z)).
\]
By (2), the right-hand side is equal to
\[
T(A_1(x),\dots,A_m(x))\wedge T(A_1(z),\dots,A_m(z)).
\]
Hence \(T_*(A_1,\dots,A_m)\) is intrinsically convex.

Conversely, assume (1). We first prove that \(T\) is monotone with respect to the product order on \(L^m\). Let \(a_s,b_s\in L\) satisfy \(a_s\leq b_s\) for every \(s=1,\dots,m\). Choose \(p,r\in X\) and \(q\in[p,r]_{\Gamma}\setminus\{p,r\}\). For each \(s=1,\dots,m\), define \(B_s:X\to L\) by
\[
B_s(q)=b_s,
\qquad
B_s(x)=a_s\quad\text{for every }x\in X\setminus\{q\}.
\]
We claim that each \(B_s\) is intrinsically convex. Let \(u,v\in X\) and \(w\in[u,v]_{\Gamma}\). If \(B_s(u)=B_s(v)=b_s\), then \(u=v=q\), and point-convexity gives \(w=q\), so the inequality is immediate. Otherwise, at least one of \(B_s(u)\) and \(B_s(v)\) is equal to \(a_s\), and therefore
\[
B_s(u)\wedge B_s(v)\leq a_s\leq B_s(w).
\]
Hence \(B_s\) is intrinsically convex. By (1), \(T_*(B_1,\dots,B_m)\) is intrinsically convex. Applying intrinsic convexity to \(p,r\) and \(q\in[p,r]_{\Gamma}\), we obtain
\[
T(a_1,\dots,a_m)
=
T_*(B_1,\dots,B_m)(p)
\wedge
T_*(B_1,\dots,B_m)(r)
\leq
T_*(B_1,\dots,B_m)(q)
=
T(b_1,\dots,b_m).
\]
Thus \(T\) is monotone.

Now choose \(p,r\in X\) and \(q\in[p,r]_{\Gamma}\setminus\{p,r\}\). Fix arbitrary elements \(a_1,\dots,a_m,b_1,\dots,b_m\in L\). For each \(s=1,\dots,m\), set \(c_s:=a_s\wedge b_s\), and define \(A_s:X\to L\) by \(A_s(p)=a_s\), \(A_s(r)=b_s\), and \(A_s(x)=c_s\) for every \(x\in X\setminus\{p,r\}\).

We claim that each \(A_s\) is intrinsically convex. Let \(u,v\in X\) and \(w\in[u,v]_{\Gamma}\). If \(u=v=p\) or \(u=v=r\), then point-convexity gives \(w=u\), and the inequality is immediate. If \(\{u,v\}=\{p,r\}\), then \(A_s(u)\wedge A_s(v)=c_s\le A_s(w)\). In all remaining cases, at least one of \(A_s(u)\) and \(A_s(v)\) is equal to \(c_s\), so \(A_s(u)\wedge A_s(v)\le c_s\le A_s(w)\). Hence \(A_s\) is intrinsically convex.

By (1), \(T_*(A_1,\dots,A_m)\) is intrinsically convex. Applying intrinsic convexity to \(p,r\) and \(q\in[p,r]_{\Gamma}\), we obtain
\[
T(a_1\wedge b_1,\dots,a_m\wedge b_m)
\geq
T(a_1,\dots,a_m)\wedge T(b_1,\dots,b_m).
\]
On the other hand, since \(a_s\wedge b_s\le a_s\) and \(a_s\wedge b_s\le b_s\) for every \(s=1,\dots,m\), the monotonicity of \(T\) gives
\[
T(a_1\wedge b_1,\dots,a_m\wedge b_m)
\leq
T(a_1,\dots,a_m)
\]
and
\[
T(a_1\wedge b_1,\dots,a_m\wedge b_m)
\leq
T(b_1,\dots,b_m).
\]
Therefore,
\[
T(a_1\wedge b_1,\dots,a_m\wedge b_m)
\leq
T(a_1,\dots,a_m)\wedge T(b_1,\dots,b_m),
\]
and equality follows.
\end{proof}

Thus, the preservation problem is controlled by the meet-compatibility of the codomain operator. The domain only provides the segments on which convexity is tested. In the classical fuzzy case \(L=[0,1]\), Theorem~\ref{teo_sharp_induced} says that an operator \(T:[0,1]^m\to[0,1]\) preserves min-based fuzzy convexity if and only if
\[
T(\min\{a_1,b_1\},\dots,\min\{a_m,b_m\})
=
\min\{T(a_1,\dots,a_m),T(b_1,\dots,b_m)\}.
\]

The condition is restrictive, but it is not vacuous. The following example shows that it is satisfied by operators other than the meet itself.

\begin{example}\label{ex:positive_operator}
In the classical fuzzy case \(L=[0,1]\), consider the binary operator \(T:[0,1]^2\to[0,1]\) defined by \(T(a,b):=\min\{a^2,b\}\). Then \(T\) is monotone in both variables and satisfies condition~\((2)\) of Theorem~\ref{teo_sharp_induced}. Indeed, for arbitrary \(a_1,a_2,b_1,b_2\in[0,1]\),
\[
\begin{aligned}
&T(\min\{a_1,b_1\},\min\{a_2,b_2\})\\
&\qquad =
\min\{\min\{a_1,b_1\}^2,\min\{a_2,b_2\}\}\\
&\qquad =
\min\{a_1^2,b_1^2,a_2,b_2\}\\
&\qquad =
\min\{T(a_1,a_2),T(b_1,b_2)\}.
\end{aligned}
\]
Hence, the pointwise operator induced by \(T\) preserves classical min-based fuzzy convexity.
\end{example}

Nevertheless, it should be noted that monotonicity alone does not guarantee this condition, even for standard aggregation operators.

\begin{example}\label{ex6}
Consider the arithmetic mean \(T:[0,1]^2\to[0,1]\), given by \(T(a,b)=(a+b)/2\). This operator is monotone in each variable. However, it does not satisfy condition (2) of Theorem~\ref{teo_sharp_induced}. Indeed, taking \((a_1,a_2)=(1,0)\) and \((b_1,b_2)=(0,1)\), we get
\(T(\min\{a_1,b_1\},\min\{a_2,b_2\})=T(0,0)=0\), whereas
\(T(a_1,a_2)\wedge T(b_1,b_2)=\frac12\wedge\frac12=\frac12\).

The failure can also be seen directly. Let \(X=\{x_1<x_2<x_3\}\), endowed with its usual order convexity, and define \(A_1,A_2:X\to[0,1]\) by
\(A_1(x_1)=1\), \(A_1(x_2)=0\), \(A_1(x_3)=0\), and
\(A_2(x_1)=0\), \(A_2(x_2)=0\), \(A_2(x_3)=1\).
Both \(A_1\) and \(A_2\) are intrinsically convex. However, their pointwise arithmetic mean satisfies
\(T_*(A_1,A_2)(x_1)=\frac12\), \(T_*(A_1,A_2)(x_2)=0\), and
\(T_*(A_1,A_2)(x_3)=\frac12\). Hence
\[
T_*(A_1,A_2)(x_2)
<
\min\{T_*(A_1,A_2)(x_1),T_*(A_1,A_2)(x_3)\}.
\]
Therefore the arithmetic mean does not preserve classical min-based fuzzy convexity.
\end{example}

The example shows that pointwise aggregation may destroy intrinsic convexity even in the classical fuzzy case. Thus, operations that are natural from a numerical or aggregative viewpoint need not be faithful to the order-theoretic structure that determines intrinsic convexity. 


\section{Symmetric convexity on hesitant fuzzy sets}\label{sec:hesitant}

In this section, we specialize the previous framework to hesitant fuzzy sets. This specialization is the point at which the distinction between intrinsic and scalarized convexity becomes most relevant: the symmetric order on hesitant fuzzy elements carries order-theoretic information that cannot be represented globally by finitely many scalar observables.

\smallskip

For consistency, we use the following notation, which is standard in the literature \citep{Alcantud,AlcantudTorra19,MerinoNavarroSalvatierraSantos2026}:
\begin{itemize}
    \item \(P^*([0,1])\) denotes the family of all nonempty subsets of \([0,1]\);
    \item \(F^*([0,1])\) denotes the family of all nonempty finite subsets of \([0,1]\).
\end{itemize}
Elements of \(P^*([0,1])\) will be called hesitant fuzzy elements (HFEs), while elements of \(F^*([0,1])\) are typical hesitant fuzzy elements (THFEs) \cite[Def.~3.1]{Bedregal2014}. In particular, \(F^*([0,1])\subseteq P^*([0,1])\).

\smallskip

We work with the symmetric lattice
\[
\bigl(P^*([0,1]),\leq^0,\wedge_0,\vee_0\bigr)
\]
introduced in \citep{JMNS22}. This lattice contains, as natural substructures, the usual fuzzy values and the interval-valued fuzzy values. Hence it allows us to treat classical fuzzy sets, interval-valued fuzzy sets, and HFSs within the same order-theoretic setting.

\smallskip

We recall that, for \(H,K\in P^*([0,1])\), one has \(H\leq^0 K\) if and only if
\[
H<K\setminus H
\quad\text{and}\quad
H\setminus K<K,
\]
where, for \(U,V\subseteq[0,1]\), \(U<V\) means that \(u<v\) for every \(u\in U\) and every \(v\in V\). As usual, this condition is interpreted vacuously when \(U=\varnothing\) or \(V=\varnothing\). Thus, if \(K\setminus H=\varnothing\), the first condition is automatically satisfied, and if \(H\setminus K=\varnothing\), the second one is automatically satisfied.

\subsection{Symmetric convexity}

\smallskip

We now specialize intrinsic lattice convexity to HFSs endowed with the symmetric lattice. Since this is the central convexity notion in the hesitant setting considered here, we refer to it as \emph{symmetric convexity}.

\begin{definition}\label{definition_symmetric-convex}
Let \((X,\Gamma)\) be a point-convex segment-generated abstract convexity space. An HFS \(A:X\to P^*([0,1])\) is said to be \emph{symmetrically convex} if
\[
A(y)\geq^0 A(x)\wedge_0 A(z)
\]
for all \(x,z\in X\) and every \(y\in[x,z]_{\Gamma}\).
\end{definition}

In the next example, we illustrate this notion.

\begin{example}\label{ex7}
Let \(X=\{x_1<x_2<x_3<x_4\}\), endowed with its usual order convexity, and define
\[
H_1=[0,0.10]\cup\{0.15\}\cup[0.34,0.40],
\]
\[
H_2=[0,0.10]\cup\{0.15\}\cup[0.20,0.27]\cup[0.58,0.63],
\]
\[
H_3=[0,0.10]\cup\{0.15\}\cup[0.20,0.40]\cup[0.71,0.75],
\]
and
\[
H_4=[0,0.10]\cup\{0.15\}\cup[0.46,0.52].
\]
Consider the HFS \(A:X\to P^*([0,1])\) given by
\[
A(x_1)=H_1,\qquad
A(x_2)=H_2,\qquad
A(x_3)=H_3,\qquad
A(x_4)=H_4.
\]

Applying the decision tree in \cite{JMNS22}, one obtains that the pair \(H_2,H_3\) has the infimum
\[
H_2\wedge_0 H_3
=
K_{23},
\qquad
K_{23}:=[0,0.10]\cup\{0.15\}\cup[0.20,0.27].
\]

Moreover, we have that
\[
H_1\wedge_0 H_3
=
H_1\wedge_0 H_4
=
H_2\wedge_0 H_4
=
K,
\qquad
K:=[0,0.10]\cup\{0.15\}.
\]
By construction, \(K\leq^0 H_\ell\) for every \(\ell=1,\ldots,4\). Hence, for each triple \(x_i<x_j<x_k\), one has
\[
A(x_i)\wedge_0 A(x_k)=K\leq^0 A(x_j).
\]
Indeed, it is satisfied that
\[
H_1\wedge_0 H_3\leq^0 H_2,\qquad
H_1\wedge_0 H_4\leq^0 H_2,\qquad
H_1\wedge_0 H_4\leq^0 H_3,\qquad
H_2\wedge_0 H_4\leq^0 H_3.
\]
If \(x_i=x_k\), then \([x_i,x_i]_{\Gamma}=\{x_i\}\) by point-convexity, and the condition is immediate. Therefore \(A\) is symmetrically convex.

\begin{figure}[htbp]
\centering
\begin{tikzpicture}[
    x=0.82cm,
    y=6.6cm,
    font=\footnotesize,
    interval/.style={line width=3pt,line cap=round},
    point/.style={circle,fill=black,inner sep=1.4pt},
    tick/.style={thin}
]

\draw[thick] (0.55,0) -- (0.55,1.00);

\foreach \t in {0,0.10,0.15,0.20,0.27,0.34,0.40,0.46,0.52,0.58,0.63,0.71,0.75,1}{
  \draw[tick] (0.50,\t) -- (0.60,\t);
}

\node[left=4pt] at (0.55,0.00) {$0$};
\node[left=4pt] at (0.55,0.10) {$0.10$};
\node[left=4pt] at (0.55,0.15) {$0.15$};
\node[left=4pt] at (0.55,0.20) {$0.20$};
\node[left=4pt] at (0.55,0.27) {$0.27$};
\node[left=4pt] at (0.55,1.00) {$1$};

\def\xa{1.55}
\def\xb{2.75}
\def\xc{3.95}
\def\xd{5.15}

\node[below] at (\xa,-0.03) {$x_1$};
\node[below] at (\xb,-0.03) {$x_2$};
\node[below] at (\xc,-0.03) {$x_3$};
\node[below] at (\xd,-0.03) {$x_4$};

\node[above] at (\xa,1.00) {$H_1$};
\node[above] at (\xb,1.00) {$H_2$};
\node[above] at (\xc,1.00) {$H_3$};
\node[above] at (\xd,1.00) {$H_4$};

\fill[gray!15] (1.28,0.00) rectangle (5.42,0.10);
\draw[dashed] (1.28,0.00) rectangle (5.42,0.10);
\foreach \x in {\xa,\xb,\xc,\xd}{
  \fill[gray!15] (\x,0.15) circle (4pt);
}

\fill[blue!12] (2.48,0.20) rectangle (4.22,0.27);
\draw[dashed,blue!70!black] (2.48,0.20) rectangle (4.22,0.27);

\draw[interval] (\xa,0.00) -- (\xa,0.10);
\node[point] at (\xa,0.15) {};
\draw[interval] (\xa,0.34) -- (\xa,0.40);

\draw[interval] (\xb,0.00) -- (\xb,0.10);
\node[point] at (\xb,0.15) {};
\draw[interval] (\xb,0.20) -- (\xb,0.27);
\draw[interval] (\xb,0.58) -- (\xb,0.63);

\draw[interval] (\xc,0.00) -- (\xc,0.10);
\node[point] at (\xc,0.15) {};
\draw[interval] (\xc,0.20) -- (\xc,0.40);
\draw[interval] (\xc,0.71) -- (\xc,0.75);

\draw[interval] (\xd,0.00) -- (\xd,0.10);
\node[point] at (\xd,0.15) {};
\draw[interval] (\xd,0.46) -- (\xd,0.52);

\node[align=left] at (7.75,0.07) {\(K=[0,0.10]\cup\{0.15\}\)};
\node[align=left] at (7.85,0.235) {\(K_{23}\setminus K=[0.20,0.27]\)};
\node[align=left] at (7.70,-0.01) {\(K\leq^0 H_1,H_2,H_3,H_4\)};

\end{tikzpicture}
\caption{A symmetrically convex hesitant fuzzy set over a finite chain.}
\label{fig:non-typical-symmetric-convexity}
\end{figure}
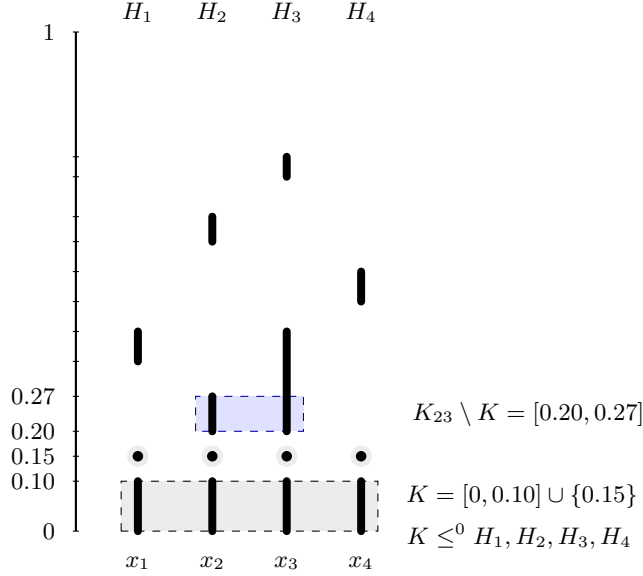
\end{example}

\begin{definition}\label{definition_simetrico_cuts}
Let \((X,\Gamma)\) be a point-convex segment-generated abstract convexity space, let \(A:X\to P^*([0,1])\), and let \(C\in P^*([0,1])\). The \emph{symmetric principal cut} of \(A\) at level \(C\) is
\[
[A]_C:=\{x\in X:\ C\leq^0 A(x)\}.
\]
\end{definition}

By Theorem~\ref{teo_princ_cuts}, it is clear that symmetric convexity is equivalent to the convexity of all symmetric principal cuts. This gives a direct cut-based interpretation, in which the levels are hesitant values ordered by \(\leq^0\).

\smallskip

The next result records how the symmetric setting restricts to two familiar cases.

\begin{theorem}\label{teo_simetrico_extension}
Let \((X,\Gamma)\) be a point-convex segment-generated abstract convexity space.
\begin{enumerate}
    \item For every \(f:X\to[0,1]\), the singleton-valued mapping \(\hat f:X\to P^*([0,1])\), given by \(\hat f(x)=\{f(x)\}\), is symmetrically convex if and only if \(f\) is min-based fuzzy convex with respect to \(\Gamma\), that is,
    \[
    f(y)\geq \min\{f(x),f(z)\}
    \]
    for all \(x,z\in X\) and every \(y\in[x,z]_{\Gamma}\).
    
    \item Let \(I([0,1])=\{[a,b]:0\le a\le b\le 1\}\). For every mapping \(A:X\to I([0,1])\), symmetric convexity is equivalent to intrinsic convexity with respect to the usual interval order, namely \([a,b]\leq_I[c,d]\) if and only if \(a\le c\) and \(b\le d\). Equivalently,
    \[
    A(y)\geq_I A(x)\wedge_I A(z)
    \]
    for all \(x,z\in X\) and every \(y\in[x,z]_{\Gamma}\), where \(\wedge_I\) denotes the componentwise meet.
\end{enumerate}
\end{theorem}

\begin{proof}
\citet{JMNS22} demonstrates that the symmetric lattice restricts to the usual order on singletons and to the usual lattice structure on closed intervals of \([0,1]\). Therefore, on singleton-valued mappings, \(\{a\}\leq^0\{b\}\) is equivalent to \(a\le b\), and \(\{a\}\wedge_0\{b\}=\{\min\{a,b\}\}\). This yields (1). For (2), the argument is analogous: on closed intervals, the restriction of \(\leq^0\) is the usual interval order \(\leq_I\), and the restriction of \(\wedge_0\) is the corresponding interval meet.
\end{proof}

\begin{corollary}\label{cor_symmetric_induced_operators}
Let \(m\geq 2\), let \((X,\Gamma)\) be a point-convex segment-generated abstract convexity space with a nontrivial segment, and let \(T:\bigl(P^*([0,1])\bigr)^m\to P^*([0,1])\) be an \(m\)-ary operator. Then the following statements are equivalent:
\begin{enumerate}
    \item for every family of symmetrically convex HFSs \(A_1,\dots,A_m:X\to P^*([0,1])\), the induced HFS \(T_*(A_1,\dots,A_m)\) is symmetrically convex;
    \item for all \(H_1,\dots,H_m,K_1,\dots,K_m\in P^*([0,1])\), one has
    \[
    T(H_1\wedge_0 K_1,\dots,H_m\wedge_0 K_m)
    =
    T(H_1,\dots,H_m)\wedge_0 T(K_1,\dots,K_m).
    \]
\end{enumerate}
\end{corollary}

\begin{proof}
This is Theorem~\ref{teo_sharp_induced} applied to the symmetric lattice \(\bigl(P^*([0,1]),\leq^0,\wedge_0,\vee_0\bigr)\).
\end{proof}

\begin{remark}\label{rem:scalar-reconstruction-layers}
Theorem~\ref{teo_simetrico_extension} clarifies the scalar complexity of the lower layers of the symmetric lattice. On singleton-valued profiles, symmetric convexity is completely reconstructed by the single observable \(\sigma(\{t\})=t\), \(t\in[0,1]\). Indeed, this observable separates the usual order on singleton values and preserves finite meets as minima. Likewise, on interval-valued profiles, symmetric convexity is completely reconstructed by the two endpoint observables \(\ell([a,b])=a\) and \(u([a,b])=b\). These observables separate the interval order \(\leq_I\) and satisfy
\(\ell(I\wedge_I J)=\min\{\ell(I),\ell(J)\}\) and
\(u(I\wedge_I J)=\min\{u(I),u(J)\}\).
Thus, finite scalar reconstruction is available on both the singleton-valued and interval-valued layers of the symmetric lattice.

\end{remark}

The next subsection shows that this finite scalar reconstruction does not extend to the genuinely hesitant regime.


\subsection{Finite scalar non-representability of the symmetric order}

\smallskip

The reconstruction results in Section~\ref{section:bridges} showed that scalar observables reconstruct intrinsic convexity only under strong order-separation and meet-compatibility requirements. We now show that, in the symmetric hesitant setting, such a reconstruction cannot be achieved globally by any finite family of scalar observables.

\smallskip

For \(n\ge 2\), let \(\mathcal I_n\) denote the set of all integer intervals \([i,j]\), with \(1\le i\le j\le n\), endowed with the order
\[
[i,j]\preceq_{\mathcal I}[k,\ell]
\Longleftrightarrow
[i,j]=[k,\ell]\ \text{or}\ j<k.
\]
This poset contains, as an induced subposet, the canonical interval order on the non-degenerate integer intervals \(1\le i<j\le n\).

\smallskip

Recall that an order-embedding of a poset \((P,\le_P)\) into a poset \((Q,\le_Q)\) is an injective map \(\Phi:P\to Q\) such that, for every \(x,y\in P\),
\[
x\le_P y
\Longleftrightarrow
\Phi(x)\le_Q\Phi(y).
\]

In view of Corollary~\ref{cor_dimension_obstruction}, it is enough to show that the symmetric lattice contains finite subposets of arbitrarily large order dimension. We do this by embedding canonical interval orders into \((F^*([0,1]),\leq^0)\). In fact, the construction below takes values in the fixed-cardinality subclass of two-point THFEs.

\begin{proposition}\label{prop_interval_embedding_final}
For every \(n\ge 2\), \((\mathcal I_n,\preceq_{\mathcal I})\) order-embeds into \(\bigl(F^*([0,1]),\leq^0\bigr)\).
\end{proposition}

\begin{proof}
Choose pairwise disjoint open intervals \(B_1,\dots,B_n\subset(0,1)\) such that every element of \(B_i\) is smaller than every element of \(B_j\) whenever \(i<j\). For each \(t\in\{1,\dots,n\}\), choose \(c_t\in B_t\), and define
\[
L_t:=B_t\cap(-\infty,c_t),
\qquad
R_t:=B_t\cap(c_t,\infty).
\]
Then \(L_t\neq\varnothing\), \(R_t\neq\varnothing\), and every element of \(L_t\) is smaller than every element of \(R_t\). Moreover, if \(i<j\), then every element of \(R_i\) is smaller than every element of \(L_j\).

Since \(\mathcal I_n\) is finite, for each \([i,j]\in\mathcal I_n\) we may choose points \(p_{ij}\in L_i\) and \(q_{ij}\in R_j\) in such a way that all chosen points are globally distinct. Define \(\Phi:\mathcal I_n\to F^*([0,1])\) by
\[
\Phi([i,j]):=\{p_{ij},q_{ij}\}.
\]
Then \(\Phi\) is injective, and distinct intervals have disjoint images.

We show that \([i,j]\preceq_{\mathcal I}[k,\ell]\) if and only if \(\Phi([i,j])\leq^0\Phi([k,\ell])\). The case \([i,j]=[k,\ell]\) is immediate. Assume then that \([i,j]\neq[k,\ell]\). Since the two images are disjoint, the defining conditions of \(\leq^0\) reduce to
\[
\Phi([i,j])\leq^0\Phi([k,\ell])
\quad\Longleftrightarrow\quad
\Phi([i,j])<\Phi([k,\ell]).
\]

If \([i,j]\preceq_{\mathcal I}[k,\ell]\), then \(j<k\). Since \(i\le j<k\le \ell\), every element of \(\Phi([i,j])\) is smaller than every element of \(\Phi([k,\ell])\). Hence \(\Phi([i,j])<\Phi([k,\ell])\), and therefore \(\Phi([i,j])\leq^0\Phi([k,\ell])\).

Conversely, assume that \(\Phi([i,j])\leq^0\Phi([k,\ell])\). Since the images are distinct and disjoint, we have \(\Phi([i,j])<\Phi([k,\ell])\). In particular, \(q_{ij}<p_{k\ell}\). We claim that \(j<k\). If \(j=k\), then \(q_{ij}\in R_j\) and \(p_{k\ell}=p_{j\ell}\in L_j\), contradicting that every element of \(L_j\) is smaller than every element of \(R_j\). If \(j>k\), then \(k<j\), and the ordering of the intervals \(B_1,\dots,B_n\) gives \(p_{k\ell}<q_{ij}\), again a contradiction. Therefore \(j<k\), and hence \([i,j]\preceq_{\mathcal I}[k,\ell]\).

Thus \(\Phi\) is an order-embedding.
\end{proof}

\begin{remark}\label{rem:two-point-obstruction}
The embedding constructed in Proposition~\ref{prop_interval_embedding_final} takes values in the subclass \(F^*_2([0,1]):=\{H\in F^*([0,1]) : \#H=2\}\). Hence the order-theoretic obstruction established below does not rely on infinite hesitant values or on unbounded local cardinalities. It already appears within pair-valued THFEs.
\end{remark}

Recall that the order dimension \(\dim(P)\) of a poset \(P\) is the least cardinal number \(d\) such that \(P\) can be order-embedded into a product of \(d\) chains.

\begin{lemma}\label{lema_finite_scores}
Let \(P\) be a poset. If there exist scalar observables \(\varphi_1,\dots,\varphi_m:P\to[0,1]\) such that
\[
x\le_P y
\Longleftrightarrow
\varphi_k(x)\le \varphi_k(y)
\quad\text{for all }k=1,\dots,m,
\]
then \(\dim(P)\le m\).
\end{lemma}

\begin{proof}
The map
\(
x\mapsto(\varphi_1(x),\dots,\varphi_m(x))
\)
is an order-embedding of \(P\) into a product of \(m\) chains. Hence \(\dim(P)\le m\).
\end{proof}

\begin{lemma}\label{cor_two_point_no_finite_representation}
The poset \(\bigl(F^*_2([0,1]),\leq^0\bigr)\) contains finite subposets of arbitrarily large order dimension. Consequently, there is no finite family of scalar observables \(\psi_1,\dots,\psi_m:F^*_2([0,1])\to[0,1]\) such that \(H\leq^0 K\) if and only if \(\psi_k(H)\leq\psi_k(K)\) for all \(k=1,\dots,m\).
\end{lemma}

\begin{proof}
By Proposition~\ref{prop_interval_embedding_final} and Remark~\ref{rem:two-point-obstruction}, for every \(n\ge 2\), the poset \((\mathcal I_n,\preceq_{\mathcal I})\) order-embeds into \(\bigl(F^*_2([0,1]),\leq^0\bigr)\).

Let \(\mathcal J_n=\{[i,j]\in\mathcal I_n: i<j\}\). This is the canonical interval order on the non-degenerate integer intervals. Since \(\mathcal J_n\) is an induced subposet of \(\mathcal I_n\), it also order-embeds into \(\bigl(F^*_2([0,1]),\leq^0\bigr)\).

By \citet[Corollary~5.2]{FurediHajnalRodlTrotter1992}, the order dimension of these canonical interval orders satisfies
\[
\dim(\mathcal J_n)
=
\lg\lg n
+
\left(\frac{1}{2}+o(1)\right)\lg\lg\lg n,
\]
where \(\lg\) denotes the binary logarithm. In particular, \(\dim(\mathcal J_n)\to\infty\) as \(n\to\infty\). Therefore \(\bigl(F^*_2([0,1]),\leq^0\bigr)\) contains finite subposets of arbitrarily large order dimension.

If a finite family \(\psi_1,\dots,\psi_m\) represented the restriction of \(\leq^0\) to \(F^*_2([0,1])\), then Lemma~\ref{lema_finite_scores} would imply that \(\dim\bigl(F^*_2([0,1]),\leq^0\bigr)\leq m\), and hence every finite subposet of \(F^*_2([0,1])\) would have order dimension at most \(m\). This contradicts the previous paragraph.
\end{proof}

\begin{theorem}\label{teo_no_global}
There is no finite family of scalar observables \(\varphi_1,\dots,\varphi_m:F^*([0,1])\to[0,1]\) such that
\[
H\leq^0 K
\Longleftrightarrow
\varphi_k(H)\leq \varphi_k(K)
\quad\text{for all }k=1,\dots,m.
\]
\end{theorem}

\begin{proof}
Otherwise, restricting such a family to \(F^*_2([0,1])\subseteq F^*([0,1])\) would contradict Lemma~\ref{cor_two_point_no_finite_representation}. 
\end{proof}

\begin{corollary}\label{corolario_no_global}
There is no finite family of observables \(\varphi_1,\dots,\varphi_m:P^*([0,1])\to[0,1]\) that represents the symmetric order \(\leq^0\) on all of \(P^*([0,1])\).
\end{corollary}

\begin{proof}
Otherwise, restricting such a family to \(F^*([0,1])\) would contradict Theorem~\ref{teo_no_global}.
\end{proof}

\begin{corollary}\label{corolario_no_scalar_convexity_reconstruction}
Let \((X,\Gamma)\) be a point-convex segment-generated abstract convexity space with a nontrivial segment. There is no finite family of observables \(\varphi_1,\dots,\varphi_m:P^*([0,1])\to[0,1]\) such that, for every HFS \(A:X\to P^*([0,1])\), \(A\) is symmetrically convex if and only if \(A\) is \(\varphi_k\)-convex for every \(k=1,\dots,m\).
\end{corollary}

\begin{proof}
If such a finite family existed, then Theorem~\ref{teo_exact_family} would imply that it separates the symmetric order \(\leq^0\) on \(P^*([0,1])\). This contradicts Corollary~\ref{corolario_no_global}. Therefore symmetric convexity cannot be reconstructed, on all HFS profiles, by any finite family of scalar observables.
\end{proof}

The previous impossibility result shows that symmetric convexity cannot be reconstructed by a finite family of scalar observables, even without assuming monotonicity. When the prescribed observables are monotone with respect to the symmetric order, this obstruction can be exhibited through an explicit three-point THFS profile.

\begin{corollary}\label{cor_three_point_scalar_failure}
Let \(X=\{x_1<x_2<x_3\}\) be endowed with its usual order convexity, and let \(\varphi_1,\dots,\varphi_m:P^*([0,1])\to[0,1]\) be any finite family of observables that are monotone with respect to \(\leq^0\). Then there exist two-point THFEs \(H,K\in F^*_2([0,1])\) such that the THFS \(A:X\to F^*_2([0,1])\) defined by \(A(x_1)=H\), \(A(x_2)=K\), and \(A(x_3)=H\) is \(\varphi_k\)-convex for every \(k=1,\dots,m\), but it is not symmetrically convex.
\end{corollary}

\begin{proof}
By Corollary~\ref{cor_two_point_no_finite_representation}, the restrictions of \(\varphi_1,\dots,\varphi_m\) to \(F^*_2([0,1])\) do not represent the order \(\leq^0\) on this subclass. Since each \(\varphi_k\) is monotone with respect to \(\leq^0\), the implication \(H\leq^0 K\Rightarrow \varphi_k(H)\leq\varphi_k(K)\) always holds for every \(k=1,\dots,m\). Therefore the failure of order representation must occur in the converse direction. Hence there exist \(H,K\in F^*_2([0,1])\) such that \(\varphi_k(H)\leq\varphi_k(K)\) for every \(k=1,\dots,m\), while \(H\not\leq^0 K\).

Define \(A:X\to F^*_2([0,1])\) by \(A(x_1)=H\), \(A(x_2)=K\), and \(A(x_3)=H\). The only nontrivial segment in \(X\) is \(x_2\in[x_1,x_3]_{\Gamma}\). For every \(k=1,\dots,m\), one has \(\varphi_k(A(x_2))=\varphi_k(K)\geq\varphi_k(H)=\min\{\varphi_k(A(x_1)),\varphi_k(A(x_3))\}\). Thus, by Theorem~\ref{teo_obs_local}, \(A\) is \(\varphi_k\)-convex for every \(k=1,\dots,m\).

However, symmetric convexity would require
\(
A(x_1)\wedge_0 A(x_3)
=
H\wedge_0 H
=
H
\leq^0
K
=
A(x_2),
\)
which is false. Therefore \(A\) is not symmetrically convex.
\end{proof}

Corollary~\ref{cor_three_point_scalar_failure} shows that the failure of finite scalar reconstruction can already be detected on a three-point domain and within the fixed-cardinality class of two-point THFEs. Hence the obstruction is not an artefact of infinite HFEs, large profiles, or varying hesitant cardinalities. This global obstruction should not be confused with the finite verifiability of a fixed profile. Indeed, if \(X=\{x_1,\ldots,x_n\}\) is finite, \((X,\Gamma)\) is a point-convex segment-generated abstract convexity space, and \(A:X\to P^*([0,1])\) is an HFS, then Theorem~\ref{teo_finite} shows that \(A\) is symmetrically convex if and only if, for every
\[
C\in\{A(x_i)\wedge_0 A(x_j):1\le i\le j\le n\},
\]
the symmetric principal cut
\[
[A]_C=\{x\in X:C\leq^0 A(x)\}
\]
is convex. Therefore, symmetric convexity of a finite HFS profile can always be verified through at most \(n(n+1)/2\) symmetric principal cuts, although no finite family of scalar observables represents the symmetric order globally.


\section{Conclusions and future research}\label{sec:conclusions}

In this article we have developed a lattice-theoretic framework for convexity on \(L\)-fuzzy sets over point-convex segment-generated abstract convexity spaces. The framework separates the segment structure of the domain from the order-theoretic structure of the codomain. Thus, linearly ordered universes are recovered by taking segments as order intervals, while ordinary segment convexity in vector spaces and non-linearly ordered betweenness structures can be treated within the same formal setting. Within this framework, two complementary perspectives have been considered: an operational perspective, formulated through \((L,\otimes)\)-convexity, and a relational perspective, formulated through preorder-based convexity.

\smallskip

The operational perspective leads naturally to intrinsic lattice convexity, obtained when the codomain operator is the meet. This notion recovers classical min-based fuzzy convexity in \eqref{eq:min-based-fuzzy-convexity}, the lattice-valued meet-based convexity considered by \citet{Maruyama2009} in \eqref{eq:maruyama-l-fuzzy-convexity}, and the meet-based instance of \(F\)-convexity in \eqref{eq:F-convexity} as particular cases. Inside this layer, we have characterized the codomain operators whose induced pointwise operators preserve intrinsic convexity. The resulting condition shows that preservation is governed by compatibility with the lattice meet; in particular, this compatibility already forces monotonicity, while monotonicity alone is not sufficient.

\smallskip

On the other hand, the relational perspective makes explicit the role of order comparisons and scalar observables. Preorder-based convexity was characterized locally, while observable convexity was shown to be a particular case induced by the total preorder associated with a numerical observable. The comparison between intrinsic and observable convexity gives the conditions under which scalar descriptions are faithful. In particular, we have shown that a single observable reconstructs intrinsic convexity only when it preserves finite meets as minima and reflects the underlying order, a requirement that is unavailable in non-chain lattices. More generally, we have demonstrated that the reconstruction by a family of observables requires both meet-compatibility and separation of the codomain order. However, for finite profiles, intrinsic convexity can still be checked through a finite family of principal cuts.

\smallskip

In the final part of the article, we have specialized the framework to HFSs. In this setting, symmetric convexity has been shown to recover the classical min-based fuzzy convexity on singleton-valued profiles and interval-valued convexity on interval-valued profiles. These two lower layers admit finite scalar reconstructions. Indeed, we have demonstrated that one observable suffices in the singleton-valued case, while the two endpoint observables reconstruct the interval-valued case. However, the genuinely hesitant regime behaves differently. By combining a reconstruction theorem for families of observables with an order-dimension obstruction, we have shown that the symmetric order cannot be globally represented by any finite family of scalar observables. This obstruction already appears within the class of two-point THFEs. Consequently, symmetric hesitant convexity cannot, in general, be reconstructed by a finite family of scalar observables. More concretely, every finite family of monotone scalar observables admits a three-point THFS profile that it classifies as convex although symmetric convexity fails. This obstruction may be relevant for the broader development of the theory of HFSs \citep{Rodriguez2016Position}. Indeed, it indicates that future research on this topic should take into account that finite scalar descriptions must be understood only as partial summaries, rather than as globally faithful substitutes for the underlying hesitant structure. This suggests that future work on HFSs should treat finite scalar 
summaries as partial descriptions rather than substitutes for the 
underlying order structure.

\smallskip

As an overview, in Table~\ref{tablageneral} we summarize how the main convexity viewpoints discussed throughout the article fit into the proposed framework and clarify the role played by each of them in the present study.

\begin{table}[htbp]
\centering
\scriptsize
\setlength{\tabcolsep}{3pt}
\renewcommand{\arraystretch}{1.18}
\caption{Main convexity viewpoints in fuzzy-set frameworks and their role in the present article.}
\label{tablageneral}
\begin{tabularx}{\textwidth}{@{}
L{0.17\textwidth}
L{0.29\textwidth}
C{0.15\textwidth}
X
@{}}
\toprule
\textbf{Convexity viewpoint} &
\textbf{Basic principle} &
\textbf{Refs.} &
\textbf{Role in the present article} \\
\midrule

Classical crisp convexity &
A set \(C\) is convex when the segment joining any two of its points remains in \(C\). &
\citep{Klee1971} &
It motivates the segment-based structure of the domain, abstracted here through point-convex segment-generated abstract convexity spaces. \\

\addlinespace

Jensen-type fuzzy convexity &
A fuzzy set \(A:X\to[0,1]\) satisfies a Jensen-type convexity condition when
\(A(\lambda x+(1-\lambda)y)\geq \lambda A(x)+(1-\lambda)A(y)\). &
\citep{AmmarMetz1992,SyauLee2003} &
It represents a stronger arithmetic formulation of fuzzy convexity, based on the concavity of the membership function, and is not generally available in arbitrary lattice-valued codomains. \\

\addlinespace

Classical min-based fuzzy convexity &
Intermediate points have membership at least the minimum of the endpoint values. &
\citep{Zadeh1965,Lowen1980} &
It is recovered when \(L=[0,1]\) and the codomain operation is \(\wedge=\min\). It is also recovered from symmetric convexity on singleton-valued HFSs. \\

\addlinespace

Lattice-valued meet-based convexity &
Membership degrees take values in a lattice, and convexity is expressed through the lattice meet and corresponding level-cut conditions. &
\citep{Maruyama2009,JanisSeseljaTepavcevic2017} &
It provides the direct background for intrinsic convexity. The article extends this viewpoint to segment-generated abstract convexity spaces and uses it as the reference notion for comparison, preservation, and scalar reconstruction. \\

\addlinespace

Operator-based convexity &
The minimum is replaced by a binary operator \(F\), leading to conditions of the form \(A(y)\geq F(A(x),A(z))\). &
\citep{Diaz2017} &
It is generalized through \((L,\otimes)\)-convexity, where the operator acts on values in a partially ordered set or lattice. \\

\addlinespace

Score-, max-, and aggregation-based convexity for THFSs &
Each typical hesitant value is reduced to a scalar summary, such as a score, a maximum, or an aggregation value. &
\citep{RashidBeg2016,JanisMontesRencova2018,Huidobro2021} &
It is reinterpreted as observable convexity. The article clarifies when such scalar descriptions preserve or fail to recover intrinsic lattice information. \\

\addlinespace

Order-based convexity &
Convexity is defined by comparing endpoint and intermediate values through a codomain order. &
\citep{Huidobro2022,Huidobro2025} &
It is recovered through preorder-based convexity and extended from linearly ordered domains to segment-generated abstract convexity spaces. \\

\addlinespace

Symmetric hesitant convexity &
HFSs are treated directly as maps into the symmetric lattice \((P^*([0,1]),\leq^0,\wedge_0,\vee_0)\). &
\citep{JMNS22} &
It is the main hesitant specialization of the framework. The article shows that no finite family of scalar observables can reconstruct symmetric convexity on all HFS profiles; this failure already appears for two-point THFEs. \\

\bottomrule
\end{tabularx}
\end{table}

\smallskip

The results raise several questions for further investigation. First, it would be interesting to further investigate codomain operators satisfying the meet-compatibility condition characterized here, especially for the concrete lattice structures that appear in fuzzy and hesitant fuzzy settings. This may clarify which natural aggregation, conjunction, or fusion mechanisms preserve intrinsic convexity and which ones necessarily destroy it. Another direction could be to relate the present framework with graded convexity structures, such as the general \((L,M)\)-fuzzy convex structures of \citet{ShiXiu2017} and the \(M\)-fuzzifying convexities induced by \(M\)-orders studied by \citet{LiShi2018}. It would also be useful to identify additional geometric assumptions on \((X,\Gamma)\) that lead to stronger preservation or reconstruction results. Finally, the symmetric hesitant setting suggests further questions, including more explicit procedures for handling symmetric principal cuts and the development of applications in decision-making models where preserving hesitant information is preferable to scalar aggregation.


\section*{Code availability}

Not applicable.

\section*{Declaration of competing interest}
 
The authors have no competing interests to declare that are relevant to the content of this article.

\section*{Acknowledgments}
   
Raquel Fernandez-Peralta was funded by the EU NextGenerationEU through the Recovery and Resilience Plan for Slovakia under the project No. 09I03-03-V04-00557. Pedro Huidobro was funded by the Spanish Ministry of Science and Innovation through project MCINN-23-PID2022-139886NB-I00, by the Asturian Agency for Science, Business Competitiveness and Innovation (SEKUENS) under Grant Agreement No. SEK-25-GRU-GIC-24-018, and by project PID2024-155289NB-I00 funded by MICIU/AEI/10.13039/501100011033 and ERDF/EU.



\bibliographystyle{plainnat}

\bibliography{references}

@ARTICLE{Alcantud,
  author  = {Alcantud, J. C. R. and Campi{\'o}n, M. J. and Indur{\'a}in, E. and Mun{\'a}rriz, A.},
  title   = {Scores of hesitant fuzzy elements revisited: `Was sind und was sollen'},
  journal = {Information Sciences},
  volume  = {648},
  year    = {2023},
  pages   = {119500},
  doi     = {10.1016/j.ins.2023.119500}
}

@ARTICLE{AlcantudTorra19,
  author  = {Alcantud, J. C. R. and Torra, V.},
  title   = {Decomposition theorems and extension principles for hesitant fuzzy sets},
  journal = {Information Fusion},
  volume  = {41},
  year    = {2018},
  pages   = {48--56},
  doi     = {10.1016/j.inffus.2017.08.005}
}

@ARTICLE{AndersonBankstonMcCluskey2021,
  author  = {Anderson, D. and Bankston, P. and McCluskey, A.},
  title   = {Convexity in topological betweenness structures},
  journal = {Topology and its Applications},
  volume  = {304},
  year    = {2021},
  pages   = {107783},
  doi     = {10.1016/j.topol.2021.107783}
}

@ARTICLE{Bedregal2014,
  author  = {Bedregal, B. and Reiser, R. H. S. and Bustince, H. and L{\'o}pez-Molina, C. and Torra, V.},
  title   = {Aggregation functions for typical hesitant fuzzy elements and the action of automorphisms},
  journal = {Information Sciences},
  volume  = {255},
  year    = {2014},
  pages   = {82--99},
  doi     = {10.1016/j.ins.2013.08.024}
}

@ARTICLE{Bustince2016,
  author  = {Bustince, H. and Barrenechea, E. and Pagola, M. and Fernandez, J. and Xu, Z. and Bedregal, B.},
  title   = {A historical account of types of fuzzy sets and their relationships},
  journal = {IEEE Transactions on Fuzzy Systems},
  volume  = {24},
  number  = {1},
  year    = {2016},
  pages   = {179--194},
  doi     = {10.1109/TFUZZ.2015.2451692}
}

@ARTICLE{Diaz2017,
  author  = {D{\'i}az, S. and Indur{\'a}in, E. and Jani{\v{s}}, V. and Llinares, J. V. and Montes, S.},
  title   = {Generalized convexities related to aggregation operators of fuzzy sets},
  journal = {Kybernetika},
  volume  = {53},
  number  = {3},
  year    = {2017},
  pages   = {383--393},
  doi     = {10.14736/kyb-2017-3-0383}
}

@ARTICLE{FarberJamison1987,
  author  = {Farber, M. and Jamison, R. E.},
  title   = {On local convexity in graphs},
  journal = {Discrete Mathematics},
  volume  = {66},
  number  = {3},
  year    = {1987},
  pages   = {231--247},
  doi     = {10.1016/0012-365X(87)90099-9}
}

@ARTICLE{Farhadinia2014,
  author  = {Farhadinia, B.},
  title   = {A series of score functions for hesitant fuzzy sets},
  journal = {Information Sciences},
  volume  = {277},
  year    = {2014},
  pages   = {102--110},
  doi     = {10.1016/j.ins.2014.02.009}
}

@INCOLLECTION{FurediHajnalRodlTrotter1992,
  author    = {F{\"u}redi, Z. and Hajnal, P. and R{\"o}dl, V. and Trotter, W. T.},
  title     = {Interval orders and shift graphs},
  booktitle = {Sets, Graphs and Numbers},
  series    = {Colloquia Mathematica Societatis J{\'a}nos Bolyai},
  volume    = {60},
  publisher = {North-Holland},
  address   = {Amsterdam},
  year      = {1992},
  pages     = {297--313}
}

@ARTICLE{Goguen1967,
  author  = {Goguen, J. A.},
  title   = {L-fuzzy sets},
  journal = {Journal of Mathematical Analysis and Applications},
  volume  = {18},
  number  = {1},
  year    = {1967},
  pages   = {145--174},
  doi     = {10.1016/0022-247X(67)90189-8}
}

@ARTICLE{GrattanGuinness1976,
  author  = {Grattan-Guinness, I.},
  title   = {Fuzzy membership mapped onto interval and many-valued quantities},
  journal = {Zeitschrift f{\"u}r Mathematische Logik und Grundlagen der Mathematik},
  volume  = {22},
  number  = {2},
  year    = {1976},
  pages   = {149--160},
  doi     = {10.1002/malq.19760220120}
}

@BOOK{Gratzer2011,
  author    = {Gr{\"a}tzer, G.},
  title     = {Lattice Theory: Foundation},
  publisher = {Birkh{\"a}user Basel},
  address   = {Basel},
  year      = {2011},
  doi       = {10.1007/978-3-0348-0018-1}
}

@ARTICLE{Huidobro2021,
  author  = {Huidobro, P. and Alonso, P. and Jani{\v{s}}, V. and Montes, S.},
  title   = {Convexity of hesitant fuzzy sets based on aggregation functions},
  journal = {Computer Science and Information Systems},
  volume  = {18},
  number  = {1},
  year    = {2021},
  pages   = {213--230},
  doi     = {10.2298/CSIS200428045H}
}

@ARTICLE{Huidobro2022,
  author  = {Huidobro, P. and Alonso, P. and Jani{\v{s}}, V. and Montes, S.},
  title   = {Convexity and level sets for interval-valued fuzzy sets},
  journal = {Fuzzy Optimization and Decision Making},
  volume  = {21},
  year    = {2022},
  pages   = {553--580},
  doi     = {10.1007/s10700-021-09376-7}
}

@ARTICLE{Huidobro2025,
  author  = {Huidobro, P. and Alonso, P. and Jani{\v{s}}, V. and Montes, S.},
  title   = {Convexity of typical hesitant fuzzy sets applied to decision-making problems},
  journal = {Computational and Applied Mathematics},
  volume  = {44},
  year    = {2025},
  pages   = {304},
  doi     = {10.1007/s40314-025-03266-z}
}

@ARTICLE{JanisMontesRencova2018,
  author  = {Jani{\v{s}}, V. and Montes, S. and Ren{\v{c}}ov{\'a}, M.},
  title   = {Convexity of hesitant fuzzy sets},
  journal = {Journal of Intelligent \& Fuzzy Systems},
  volume  = {34},
  number  = {4},
  year    = {2018},
  pages   = {2099--2112},
  doi     = {10.3233/JIFS-162204}
}

@ARTICLE{JMNS22,
  author  = {Jara, P. and Merino, L. and Navarro, G. and Santos, E.},
  title   = {A lattice structure on hesitant fuzzy sets},
  journal = {IEEE Transactions on Fuzzy Systems},
  volume  = {31},
  number  = {6},
  year    = {2023},
  pages   = {2018--2028},
  doi     = {10.1109/TFUZZ.2022.3217400}
}

@ARTICLE{Jahn1975,
  author  = {Jahn, K. U.},
  title   = {Intervall-wertige Mengen},
  journal = {Mathematische Nachrichten},
  volume  = {68},
  number  = {1},
  year    = {1975},
  pages   = {115--132},
  doi     = {10.1002/mana.19750680109}
}

@ARTICLE{LobilloMerinoNavarroSantos2021,
  author  = {Lobillo, F. J. and Merino, L. and Navarro, G. and Santos, E.},
  title   = {Induced triangular norms and negations on bounded lattices},
  journal = {IEEE Transactions on Fuzzy Systems},
  volume  = {29},
  number  = {7},
  year    = {2021},
  pages   = {1802--1814},
  doi     = {10.1109/TFUZZ.2020.2985337}
}

@ARTICLE{Maruyama2009,
  author  = {Maruyama, Y.},
  title   = {Lattice-valued fuzzy convex geometry},
  journal = {RIMS K{\^o}ky{\^u}roku},
  number  = {1641},
  year    = {2009},
  pages   = {22--37}
}

@MISC{MerinoNavarroSalvatierraSantos2026,
  author       = {Merino, L. and Navarro, G. and Salvatierra, C. and Santos, E.},
  title        = {An order-oriented approach to scoring hesitant fuzzy elements},
  howpublished = {arXiv preprint arXiv:2602.16827},
  year         = {2026},
  doi          = {10.48550/arXiv.2602.16827}
}

@ARTICLE{MerinoNavarroSantos2022,
  author  = {Merino, L. and Navarro, G. and Santos, E.},
  title   = {Induced operators on bounded lattices},
  journal = {Information Sciences},
  volume  = {608},
  year    = {2022},
  pages   = {114--136},
  doi     = {10.1016/j.ins.2022.06.033}
}

@ARTICLE{RashidBeg2016,
  author  = {Rashid, T. and Beg, I.},
  title   = {Convex hesitant fuzzy sets},
  journal = {Journal of Intelligent \& Fuzzy Systems},
  volume  = {30},
  number  = {5},
  year    = {2016},
  pages   = {2791--2796},
  doi     = {10.3233/IFS-152057}
}

@ARTICLE{Rodriguez2016Position,
  author  = {Rodr{\'i}guez, R. M. and Bedregal, B. and Bustince, H. and Dong, Y. C. and Farhadinia, B. and Kahraman, C. and Mart{\'i}nez, L. and Torra, V. and Xu, Y. J. and Xu, Z. S. and Herrera, F.},
  title   = {A position and perspective analysis of hesitant fuzzy sets on information fusion in decision making. Towards high quality progress},
  journal = {Information Fusion},
  volume  = {29},
  year    = {2016},
  pages   = {89--97},
  doi     = {10.1016/j.inffus.2015.11.004}
}

@PHDTHESIS{Sambuc75,
  author  = {Sambuc, R.},
  title   = {Fonctions \(\phi\)-floues: Application {\`a} l'aide au diagnostic en pathologie thyro{\"i}dienne},
  school  = {Universit{\'e} de Marseille},
  year    = {1975},
  type    = {Ph.D. Thesis},
  address = {France}
}

@ARTICLE{ShiXiu2017,
  author  = {Shi, F.-G. and Xiu, Z.-Y.},
  title   = {$(L,M)$-fuzzy convex structures},
  journal = {Journal of Nonlinear Sciences and Applications},
  volume  = {10},
  number  = {7},
  year    = {2017},
  pages   = {3655--3669},
  doi     = {10.22436/jnsa.010.07.25}
}

@ARTICLE{Torra2010,
  author  = {Torra, V.},
  title   = {Hesitant fuzzy sets},
  journal = {International Journal of Intelligent Systems},
  volume  = {25},
  number  = {6},
  year    = {2010},
  pages   = {529--539},
  doi     = {10.1002/int.20418}
}

@BOOK{vandeVel1993,
  author    = {van de Vel, M. L. J.},
  title     = {Theory of Convex Structures},
  series    = {North-Holland Mathematical Library},
  volume    = {50},
  publisher = {North-Holland},
  address   = {Amsterdam},
  year      = {1993}
}

@ARTICLE{XuXia2011,
  author  = {Xu, Z. and Xia, M.},
  title   = {Distance and similarity measures for hesitant fuzzy sets},
  journal = {Information Sciences},
  volume  = {181},
  number  = {11},
  year    = {2011},
  pages   = {2128--2138},
  doi     = {10.1016/j.ins.2011.01.028}
}

@ARTICLE{Zadeh1965,
  author  = {Zadeh, L. A.},
  title   = {Fuzzy sets},
  journal = {Information and Control},
  volume  = {8},
  number  = {3},
  year    = {1965},
  pages   = {338--353},
  doi     = {10.1016/S0019-9958(65)90241-X}
}

@ARTICLE{Zadeh1971,
  author  = {Zadeh, L. A.},
  title   = {Quantitative fuzzy semantics},
  journal = {Information Sciences},
  volume  = {3},
  number  = {2},
  year    = {1971},
  pages   = {159--176},
  doi     = {10.1016/S0020-0255(71)80004-X}
}

@ARTICLE{Klee1971,
  author  = {Klee, Victor},
  title   = {What Is a Convex Set?},
  journal = {The American Mathematical Monthly},
  volume  = {78},
  number  = {6},
  pages   = {616--631},
  year    = {1971},
  doi     = {10.1080/00029890.1971.11992815}
}

@ARTICLE{AmmarMetz1992,
  author  = {Ammar, E. E. and Metz, J.},
  title   = {On Fuzzy Convexity and Parametric Fuzzy Optimization},
  journal = {Fuzzy Sets and Systems},
  volume  = {49},
  number  = {2},
  pages   = {135--141},
  year    = {1992},
  doi     = {10.1016/0165-0114(92)90319-Y}
}

@INPROCEEDINGS{SyauLee2003,
  author    = {Syau, Yu-Ru and Lee, E. Stanley},
  title     = {Fuzzy Convexity with Application to Fuzzy Decision Making},
  booktitle = {Proceedings of the 42nd IEEE Conference on Decision and Control},
  pages     = {5221--5226},
  year      = {2003},
  doi       = {10.1109/CDC.2003.1272466}
}

@ARTICLE{JanisSeseljaTepavcevic2017,
  author  = {Jani{\v{s}}, Vladim{\'i}r and {\v{S}}e{\v{s}}elja, Branimir and Tepav{\v{c}}evi{\'c}, Andreja},
  title   = {Poset valued convexities},
  journal = {Information Sciences},
  volume  = {406--407},
  pages   = {208--215},
  year    = {2017},
  doi     = {10.1016/j.ins.2017.04.031}
}

@ARTICLE{Lowen1980,
  author  = {Lowen, Robert},
  title   = {Convex fuzzy sets},
  journal = {Fuzzy Sets and Systems},
  volume  = {3},
  pages   = {291--310},
  year    = {1980},
  doi     = {10.1016/0165-0114(80)90025-1}
}

@ARTICLE{PangShi2017,
  author  = {Pang, Bin and Shi, Fu-Gui},
  title   = {Subcategories of the category of \(L\)-convex spaces},
  journal = {Fuzzy Sets and Systems},
  volume  = {313},
  pages   = {61--74},
  year    = {2017},
  doi     = {10.1016/j.fss.2016.02.014}
}

@ARTICLE{LiShi2018,
  author  = {Li, Erqiang and Shi, Fu-Gui},
  title   = {Some properties of \(M\)-fuzzifying convexities induced by \(M\)-orders},
  journal = {Fuzzy Sets and Systems},
  volume  = {350},
  pages   = {41--54},
  year    = {2018},
  doi     = {10.1016/j.fss.2018.05.008}
}

@ARTICLE{JanisKralRencova2013,
  author  = {Jani{\v{s}}, V. and Kr{\'a}{\v{l}}, P. and Ren{\v{c}}ov{\'a}, M.},
  title   = {Aggregation operators preserving quasiconvexity},
  journal = {Information Sciences},
  volume  = {228},
  year    = {2013},
  pages   = {37--44},
  doi     = {10.1016/j.ins.2012.12.003}
}



\end{document}